\documentclass[a4paper,10pt]{article}
\usepackage[utf8]{inputenc}
\usepackage[english]{babel}
\usepackage{amsmath,amsfonts,amssymb,amsthm}
\usepackage{graphicx}
\usepackage{xspace}
\usepackage{tikz}
\usepackage{comment}
\usepackage{authblk}
\usepackage{cellspace}
\usepackage{vmargin}

\setmarginsrb{3.8cm}{2cm}{3.8cm}{2cm}{0cm}{2cm}{0cm}{1cm}

\theoremstyle{plain}
\newtheorem{proposition}{Proposition}

\newtheorem{lemma}[proposition]{Lemma}

\newtheorem{theorem}[proposition]{Theorem}

\theoremstyle{definition}

\newtheorem{example}[proposition]{Example}
\newtheorem{remark}[proposition]{Remark}

\title{Constant $2$-labellings and an application to $(r,a,b)$-covering codes}
\author[1]{Sylvain Gravier}
\author[1,2]{\'Elise Vandomme}

\affil[1]{University of Grenoble, France}
\affil[2]{University of Liege, Belgium}

\usepackage{hyperref}
\begin{document}

\maketitle

\begin{abstract}
	We introduce the concept of constant $2$-labelling of a weighted graph and show how it can be used to obtain perfect weighted coverings. Roughly speaking, a constant $2$-labelling of a weighted graph is a $2$-colouring of its vertex set which preserves the sum of the weights of black vertices under some automorphisms. We study this problem on four types of weighted cycles.  Our results on cycles allow us to determine $(r,a,b)$-codes in $\mathbb{Z}^2$ whenever $|a-b|>4$, $r\ge2$ and we give the precise values of $a$ and $b$.  This is a refinement of Axenovich's theorem proved in 2003. 
\end{abstract}

\section*{Introduction}
Constant $2$-labellings are particular $2$-colourings of weighted graphs. For every composition of the colouring with an automorphism of a given group, the sum of the weights of the black vertices must be equal to a constant that depends on the colour of a given particular vertex.

The motivation about introducing constant $2$-labellings comes from covering problems in graphs. 
Covering and packing problems are traditional issues in graph theory. A natural packing problem in a graph is to determine the maximal number of non-intersecting identical balls that can be placed in the graph; a covering problem is for example to determine the minimal number of $r$-balls that can be placed in such a way that every vertex of the graph is contained in at least one of them. Packing problems are fundamental in ``error correction'' while covering problems have application in mobile network. See~\cite{Cohen--Honkala--Litsyn--Lobstein--1997} for many bibliographic pointers. 

We consider in this paper coverings with balls of constant radius that satisfy multiplicity conditions. For positive integers  $r, a, b$, an \emph{$(r,a,b)$-covering code} or simply \emph{$(r,a,b)$-code} of a graph $G=(V,E)$ is a set $S\subseteq V$ of vertices such that every element of $S$ belongs to exactly $a$ balls of radius $r$ with elements of $S$ as centers and every element of $V\setminus S$ belongs to exactly $b$ balls 
of radius $r$ with elements of $S$ as centers. We can view an $(r,a,b)$-code as a particular colouring $c$ with two colors, black and white, where the black vertices are the elements of the code. Such codes are also known as $(a,b)$-codes of radius $r$~\cite{Gravier--Lacroix--Slimani--2013}, as $(r,a,b)$-isotropic colourings~\cite{Axenovich--2003} or as perfect colourings~\cite{Puzynina--2009}.

The notion of $(r,a,b)$-codes generalizes the notion of domination and perfect codes in graphs. Perfect codes were introduced in terms of graphs by Biggs~\cite{Biggs--1973}. An \emph{$r$-perfect code} of a graph $G=(V,E)$ is a subset $C\subseteq V$ with the property that each vertex is within distance $r$ of exactly one vertex of $C$. Hence, an $r$-perfect code is an $(r,1,1)$-code.  
Kratochv\'il~\cite{Kratochvil--1988} showed that the problem of finding an $r$-perfect code in  graphs is NP-complete. 

Perfect codes have also been studied in infinite graphs. For example, Golomb and Welsh \cite{Golomb--Welch--1968,Golomb--Welch--1970} considered the multidimensional rectangular grid $\mathbb{Z}^d$. 
They proved the existence of $1$-perfect codes, i.e., $(1,1,1)$-codes, in $\mathbb{Z}^d$. Such codes can be considered as periodic tilings of the grid $\mathbb{Z}^n$ by balls of radius $1$. Moreover, the authors conjectured that there do not exist $r$-perfect codes with $r> 1$ in $\mathbb{Z}^d$~\cite{Golomb--Welch--1968,Golomb--Welch--1970}. Horak~\cite{Horak--2009} wrote a survey of results that support the conjecture. For more information about perfect codes, see~\cite[Chapter 11]{Cohen--Honkala--Litsyn--Lobstein--1997}.

The $(r,a,b)$-codes have already been studied in some graphs under the name of \emph{weighted covering codes} by Cohen et al.~\cite{Cohen--Litsyn--Honkala--Mattson--1995}. Their work corresponds to a study of these codes in the Hamming metric. For a subset $C$ of vertices, they attach weights to different layers of the Hamming sphere and they consider weighted spheres centred at vertices of $C$. If several such spheres intersect in a vertex, they define the density of each vertex as the sum of the weights of the corresponding layers. The set $C$ is called a \emph{weighted covering} if the density at each vertex is at least one. When the density is exactly equal to one for all vertices, then $C$ is called a \emph{perfect weighted covering}. If the radius is equal to $1$, a $(1,a,b)$-code is exactly a perfect weighted covering of radius one with weight $\left(\frac{b-a+1}{b},\frac{1}{b}\right)$. For more details see~\cite[Chapter 13]{Cohen--Honkala--Litsyn--Lobstein--1997}.

While Cohen et al.~\cite{Cohen--Litsyn--Honkala--Mattson--1995} studied weighted codes in Hamming metric, Telle considered a particular case of these codes in graphs in general~\cite{Telle--1994}. For a subset $C$ of vertices, he defines the state of a vertex $u\in C$ by
$$state(u)=\left\{\begin{array}{ll}
\sigma_i & \text{ if }u\in C\text{ and }|N(u)\cap C|=i\\
\rho_i & \text{ if }u\not\in C\text{ and }|N(u)\cap C|=i.
\end{array}\right.$$
Then many properties of vertex subsets can be defined by allowing only a specific set $L$ of states. For instance, the set $C$ is a dominating set if the state $\rho_0$ is not allowed.  In this setting, $(1,a,b)$-codes are equivalent to $[\sigma_{a-1},\rho_b]$-dominating sets. Telle~\cite{Telle--1994} proved that the following decidability problem was NP-complete: ``Is it possible to decide whether a graph has an $[\sigma_a,\rho_b]$-dominating set ?''. The problem is still NP-complete when restricted to planar bipartite graphs of maximum degree three.

The particular case where the radius is $1$ has been studied a lot. For instance, $(1,a,b)$-codes are equivalent to equitable partitions with two cells~\cite[Chapter~5]{Godsil--1993}. In the multidimensional grid, which corresponds to the Lee metric with an infinite alphabet, $(1,a,b)$-codes were studied by Dorbec et al.~\cite{Dorbec--Gravier--Honkala--Mollard--2009} and Gravier et al.~\cite{Gravier--Mollard--Payan--1999}. For instance, the existence of $(1,2,1)$-codes\footnote{The reader may notice a distinction of notation between this paper and~\cite{Dorbec--Gravier--Honkala--Mollard--2009}. They write $(a,b)$-codes for what we denote $(1,a+1,b)$-codes as they consider open neighbourhoods and we consider closed neighbourhoods.} in $\mathbb{Z}^d$ is proved in both papers~\cite{Dorbec--Gravier--Honkala--Mollard--2009, Gravier--Mollard--Payan--1999}. In~\cite[Theorem~4]{Dorbec--Gravier--Honkala--Mollard--2009}, Dorbec et al.~present a method to construct $(1,a,b)$-codes in $\mathbb{Z}^d$. This method is based on a one-dimensional pattern of finite length that is extended by translations to colour $\mathbb{Z}^d$. Hence, the code obtained satisfies  periodic properties. 

\begin{theorem}[Dorbec et al.~\cite{Dorbec--Gravier--Honkala--Mollard--2009}]\label{thm:dorbec}
Assume that $1\le k\le n, 1\le d$ and 
$$A=\{a_1,a_2,\ldots,a_k\}\subseteq \mathbb{Z}_n \text{ (where $a_i\ne a_j$, when $i\ne j$)}$$
and $w_1,\ldots,w_d$ are (not necessarily distinct) elements of $\mathbb{Z}_n$. Consider the sums $a_i+w$ and the differences $a_j-w_j$. If these $2kd$ elements take each value in $A$ exactly $a$ times and each value in $\mathbb{Z}_n\setminus A$ exactly $b$ times, then the set 
$$C=\{(x_1,\ldots,x_d)\in\mathbb{Z}^d\mid x_1w_1+\cdots+x_dw_d\in A\}$$
is a $(1,a+1,b)$-code of $\mathbb{Z}^d$.
\end{theorem}

In the two-dimensional grid, i.e., the usual infinite grid, Puzynina studied the periodicity of $(r,a,b)$-codes. For $r=1$, there exist non-periodic $(r,a,b)$-codes but, all of them can be obtained from periodic ones~\cite{Puzynina--2004}. That is to say, if $r, a, b$ are such that there exists an $(r,a,b)$-code, then there exists a periodic colouring that is an $(r,a,b)$-code.  Moreover when $r\ge2$, the author proved that every $(r,a,b)$-code is periodic~\cite{Puzynina--2009}. The notion of constant $2$-labellings comes up as a natural translation of the periodicity of $(r,a,b)$-codes in the infinite grid $\mathbb{Z}^2$.

When the difference between $a$ and $b$ is large enough, the precise type of the periodic colouring is known. 
 
\begin{theorem}[Axenovich \cite{Axenovich--2003}]\label{thm_Axe}
If a colouring is an $(r,a,b)$-code of $\mathbb{Z}^2$ with $r\ge2$ and $|a-b|>4$, then it is one of the following diagonal colourings 1--5:
\begin{enumerate}
\item $q$-periodic colouring where $q\in\{r,r+1\}$ is odd and the monochromatic diagonals are parallel.
\item $q$-anti-periodic colouring where $q\in\{r,r+1\}$ is even.
\item $q$-periodic colouring where $q\in\{r,r+1\}$ is even and for all horizontal or vertical interval $I$ of length
$p$ the number of black vertices from the even sublattice and from the odd sublattice is the same.
\item $(2r+1)$-periodic colouring and for all horizontal or vertical interval $I$ of length
$p$ the number of black vertices from the even sublattice and from the odd sublattice is the same.
\item $2$-periodic or $3$-periodic colouring.
\end{enumerate}
\end{theorem}

Using Axenovich's characterization in terms of diagonal colouring of all $(r,a,b)$-codes in $\mathbb{Z}^2$ with $r\ge2$ and $|a-b|>4$ (Theorem~\ref{thm_Axe}), we show that the existence of $(r,a,b)$-codes in the infinite grid is linked with the existence of constant $2$-labellings in particular cycles. It turns out that studying only four types of weighted cycles is sufficient to characterize all $(r,a,b)$-codes with $|a-b|>4$ and to determine explicitly the possible values taken by the constants $a$ and $b$. Hence, we obtain a refinement of Axenovich's theorem.

\smallskip

\paragraph{Outline} This paper is organized as follows. The first section is dedicated to the presentation of constant $2$-labellings of weighted graphs in a general framework. Then we focus on the constant $2$-labellings in four types of weighted cycles. In Section~\ref{sec:projection_folding}, we present projection and folding techniques that link constant $2$-labellings to $(r,a,b)$-codes. Hopefully, these techniques can be applied to other problems involving periodic tilings. In Section~\ref{sec:application_rabc}, we apply the projection and folding method to obtain all possible values of constants $a$ and $b$ such that there exist $(r,a,b)$-codes of $\mathbb{Z}^2$ with $|a-b|>4$ and $r\ge2$. 
Note that to apply this method, the colouring of the grid must satisfy some specific properties.  Finally, we suggest directions for future work.

\section{Constant $2$-labellings}\label{main}

Given a graph $G=(V,E)$, a particular vertex $v\in V$, a map $w:V\to\mathbb{R}$ and a subgroup $A$ of the set $Aut(G)$
of all automorphisms of $G$, a \emph{constant 2-labelling} of $G$ is a mapping $c:V\to\{{\tt 0,1}\}$ such that there exist constants $a$ and $b$ satisfying
$$a=\sum_{\{u\in V \mid c(\xi(u))={\tt1}\}} w(u), \quad \forall\xi\in A_{\tt 1} \;\text{ and }\;
b= \sum_{\{u\in V \mid c(\xi'(u))={\tt1}\}} w(u),
\quad \forall\xi'\in A_{\tt 0}.$$
where $A_{\tt 1}=\{\xi\in A\mid c(\xi(v))={\tt1}\}$ and $A_{\tt 0}=\{\xi\in A\mid c(\xi(v))={\tt0}\}$. 

\begin{example} let $G=(V,E)$ be the graph with $V=\{v_0,\ldots,v_4\}$ represented in Figure~\ref{fig1}. 
Take $v=v_0, A=Aut(G), w:V\to\mathbb{R}$ and $c:V\to\{{\tt 0,1}\}$ defined by
$w(v_0)=3$, \mbox{$w(v_1)=w(v_3)=2$,} $w(v_2)=w(v_4)=5$ and $c(v_0)=c(v_3)=c(v_4)={\tt 0}, 
c(v_1)=c(v_2)={\tt1}$. It is clear that $c$ is a constant 2-labelling since 
$A$ contains only two automorphisms, $id$ and 
$$ \sigma :  v_0\mapsto v_0;\; v_1 \mapsto v_4;\; v_2\mapsto v_3;\; v_3\mapsto v_2;\; v_4\mapsto v_1.$$
\end{example}

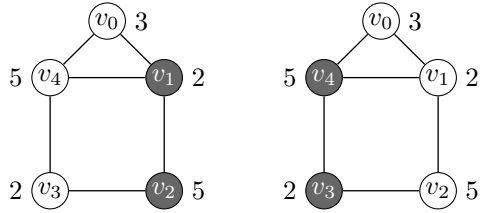
\begin{figure}[h!tbp]
  \begin{center}
  \begin{tikzpicture}[scale=1.5]
    \tikzstyle{every node}=[shape=circle,fill=white,draw=black,minimum size=8pt,inner sep=1pt]
    \tikzstyle{black}=[shape=circle,fill=black!60,draw=black,minimum size=8pt,inner sep=1pt]
    \tikzstyle{w}=[fill=none,draw=none,minimum size=10pt,inner sep=1pt]
    \node(a) at (0.5,1.5) {$v_0$};		\node[w] at (0.8,1.5) {3};
    \node[black](b) at (1,1) {{\color{white}$v_1$}};	\node[w] at (1.3,1) {2};
    \node[black](c) at (1,0) {{\color{white}$v_2$}};	\node[w] at (1.3,0) {5};
    \node(d) at (0,0) {$v_3$};		\node[w] at (-0.3,0) {2};
    \node(e) at (0,1) {$v_4$};		\node[w] at (-0.3,1) {5};
    \tikzstyle{every path}=[color =black, line width = 0.5 pt]
    \tikzstyle{every node}=[shape=circle,minimum size=20pt,inner sep=2pt]
    \draw (a) to node {}  (b);
    \draw (c) to node {}  (b);
    \draw (e) to node {}  (b);
    \draw (c) to node {}  (d);
    \draw (e) to node {}  (d);
    \draw (e) to node {}  (a);
  \end{tikzpicture}$\qquad$
  \begin{tikzpicture}[scale=1.5]
    \tikzstyle{every node}=[shape=circle,fill=white,draw=black,minimum size=8pt,inner sep=1pt]
    \tikzstyle{black}=[shape=circle,fill=black!60,draw=black,minimum size=8pt,inner sep=1pt]
    \tikzstyle{w}=[fill=none,draw=none,minimum size=10pt,inner sep=1pt]
    \node(a) at (0.5,1.5) {$v_0$};	\node[w] at (0.8,1.5) {3};
    \node(b) at (1,1) {$v_1$};	\node[w] at (1.3,1) {2};
    \node(c) at (1,0) {$v_2$};		\node[w] at (1.3,0) {5};
    \node[black](d) at (0,0) {{\color{white}$v_3$}};	\node[w] at (-0.3,0) {2};
    \node[black](e) at (0,1) {{\color{white}$v_4$}};		\node[w] at (-0.3,1) {5};
    \tikzstyle{every path}=[color =black, line width = 0.5 pt]
    \tikzstyle{every node}=[shape=circle,minimum size=20pt,inner sep=2pt]
    \draw (a) to node {}  (b);
    \draw (c) to node {}  (b);
    \draw (e) to node {}  (b);
    \draw (c) to node {}  (d);
    \draw (e) to node {}  (d);
    \draw (e) to node {}  (a);
  \end{tikzpicture}
  \end{center}
  \caption{A colouring of a graph $G$ (on the left) and its composition with the automorphism $\sigma$ (on the right).}\label{fig1}
\end{figure}

We can make some straightforward observations about constant $2$-labellings. The following proposition allows us to consider either a colouring $c$ or its complement colouring $\overline{c}$.

\begin{proposition}[Complementary property]\label{prop_compl}
 Let $G=(V,E)$ be a weighted graph, $w:V\to\mathbb{R}$ be the weight map, $v\in V$ and $A\leq Aut(G)$. 
 Set $\omega:=\sum_{u\in V} w(u)$. 
 A colouring $c$ is a constant $2$-labelling of $G$ with respective constants $a$ and $b$ if and only if
 the colouring $\overline{c}$ is a constant $2$-labelling with respective constants $\omega-b$ and $\omega-a$.
\end{proposition}

 If $c$ is monochromatic black, then the constants are such that $a=\sum_{u\in V} w(u)$ and $b$ is not defined. Otherwise $c$ is white, $a$ is not defined and $b=0$.
It is clear that, for a weighted graph $G=(V,E)$ with $v\in V$, any trivial colouring of $V$ is a constant $2$-labelling for any weight map and any subgroup of $Aut(G)$. Such constant $2$-labellings are called \emph{trivial}.
The definition of constant $2$-labellings gives rise to the natural question: whether there exists non-trivial constant $2$-labellings for some classes of weighted graphs. We answer that question in the case of four types of weighted cycles in the next subsection.

\begin{remark}
Consider the complete graph $K_n$ and let $w:V(K_n)\to\mathbb{R}$, $v\in V(K_n)$, $A=Aut(K_n)$. It is straightforward to show that there exists a non-trivial constant $2$-labelling of $K_n$
if and only if $w(v_1)=w(v_2)$ for all $v_1, v_2 \in V\setminus\{v\}$. See \cite[Remark~6.3.]{Vandomme--thesis} for details.

%
\end{remark}

\subsection{Constant $2$-labellings in particular weighted cycles}

We consider some particular weighted cycles $\mathcal{C}_p$ 
with at most 4 different weights on the vertices $0,\ldots,p-1$. If the weights are $w(0) ,\ldots, w(p-1)$, then we represent the cycle by the word $w(0)\ldots w(p-1)$. We will use the letters $z,x,y$ and $t$ to denote the weights of vertices. For instance, the cycle depicted in Figure~\ref{fig:cycle_0} is represented by the word $zx^{p-1}$.

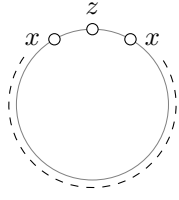
\begin{figure}[h!tbp]
  \begin{center}
  \begin{tikzpicture}[scale=1]
	  \path[draw,help lines] (0,0) circle (1 cm);
	  \tikzstyle{trait}=[dashed]
	  \draw[trait] (-215:1.1) arc (-215:35:1.1cm);
	  \tikzstyle{every node}=[shape=circle,fill=white,draw=black,minimum size=0.5pt,inner sep=1.5pt]
	  \tikzstyle{empty}=[fill=none,draw=none,minimum size=0.5pt,inner sep=1.5pt]
	  \node at (60:1)[label=right:$x$] {};
	  \node at (90:1) [label=above:$z$]{};
	  \node at (120:1)[label=left:$x$] {};
  \end{tikzpicture}
  \end{center}
  \vspace{-0.3cm}
  \caption{Weighted cycle $\mathcal{C}_p$ of Type 0.}\label{fig:cycle_0}
\end{figure}

We restrict our study of constant $2$-labellings to only four types of weighted cycles (see Figure \ref{figtypes}) and we set $v:=0$ and $A:=\{\mathcal{R}_k\mid k\in \mathbb{Z}\}$. This restriction is due to the initial motivation behind this work: to know the possible values of the constants in $(r,a,b)$-code of the infinite grid. The four types of weighted cycles on $p$ vertices depend on $p\bmod{4}$.
\begin{itemize}
\item  If $p\equiv 1 \pmod{4}$,  the cycle represented by the word $z(xy)^\frac{p-1}{4}(yx)^\frac{p-1}{4}$ is called \emph{Type1mod}. 
\item  If $p\equiv 2 \pmod{4}$, the cycle represented by the word $z(xy)^\frac{p-2}{4}t(yx)^\frac{p-2}{4}$ is called \emph{Type2mod}. 
\item  If $p\equiv 3 \pmod{4}$,  the cycle represented by the word $z(xy)^\frac{p-3}{4}xx(yx)^\frac{p-3}{4}$ is called \emph{Type3mod}.
\item  If $p\equiv 0 \pmod{4}$,  the cycle represented by the word $z(xy)^\frac{p-4}{4}xtx(yx)^\frac{p-4}{4}$ is called \emph{Type4mod}. 
\end{itemize}
Hence, we only consider weighted cycles with an axial symmetry in the distribution of weights. It seems to play an important role for the existence of constant $2$-labellings. For instance, a weighted cycle $\mathcal{C}_p$ represented by the word $z(xy)^{\frac{p-1}{2}}$ with $x\ne y$, has only monochromatic colourings as constant $2$-labellings. See \cite[Lemma~B.3.]{Vandomme--thesis} for a proof.

\begin{figure}[h!tbp]
  \begin{center}
  \begin{tabular}{c c c}
  Type1mod: $z(xy)^\frac{p-1}{4}(yx)^\frac{p-1}{4}$&$\quad$&Type3mod: $z(xy)^\frac{p-3}{4}xx(yx)^\frac{p-3}{4}$\\
  \begin{tikzpicture}[scale=1]
	  \path[draw,help lines] (0,0) circle (1 cm);
	  \tikzstyle{trait}=[dashed]
	  \draw[trait] (165:1.1) arc (165:210:1.1cm);
	  \draw[trait] (-30:1.1) arc (-30:15:1.1cm);
	  \tikzstyle{every node}=[shape=circle,fill=white,draw=black,minimum size=0.5pt,inner sep=1.5pt]
	  \node at (30:1)[label=right:$y$] {};
	  \node at (60:1)[label=right:$x$] {};
	  \node at (90:1) [label=above:$z$]{};
	  \node at (120:1)[label=left:$x$] {};
	  \node at (150:1) [label=left:$y$]{};
	  \node at (225:1) [label=left:$x$]{};
	  \node at (255:1) [label=below:$y$]{};
	  \node at (285:1) [label=below:$y$]{};
	  \node at (315:1) [label=right:$x$]{};
  \end{tikzpicture}
 &$\quad$ &\begin{tikzpicture}[scale=1]
	  \path[draw,help lines] (0,0) circle (1 cm);
	  \tikzstyle{trait}=[dashed]
	  \draw[trait] (165:1.1) arc (165:240:1.1cm);
	  \draw[trait] (-60:1.1) arc (-60:15:1.1cm);
	  \tikzstyle{every node}=[shape=circle,fill=white,draw=black,minimum size=0.5pt,inner sep=1.5pt]
	  \node at (30:1)[label=right:$y$] {};
	  \node at (60:1)[label=right:$x$] {};
	  \node at (90:1) [label=above:$z$]{};
	  \node at (120:1)[label=left:$x$] {};
	  \node at (150:1) [label=left:$y$]{};
	  \node at (255:1) [label=below:$x$]{};
	  \node at (285:1) [label=below:$x$]{};
  \end{tikzpicture}\\
\\
  Type2mod: $z(xy)^\frac{p-2}{4}t(yx)^\frac{p-2}{4}$ &$\quad$&Type4mod: $z(xy)^\frac{p-4}{4}xtx(yx)^\frac{p-4}{4}$\\
  \begin{tikzpicture}[scale=1]
	  \path[draw,help lines] (0,0) circle (1 cm);
	  \tikzstyle{trait}=[dashed]
	  \draw[trait] (165:1.1) arc (165:195:1.1cm);
	  \draw[trait] (-15:1.1) arc (-15:15:1.1cm);
	  \tikzstyle{every node}=[shape=circle,fill=white,draw=black,minimum size=0.5pt,inner sep=1.5pt]
	  \node at (30:1)[label=right:$y$] {};
	  \node at (60:1)[label=right:$x$] {};
	  \node at (90:1) [label=above:$z$]{};
	  \node at (120:1)[label=left:$x$] {};
	  \node at (150:1)[label=left:$y$] {};
	  \node at (210:1) [label=left:$x$]{};
	  \node at (240:1) [label=left:$y$]{};
	  \node at (270:1) [label=below:$t$]{};
	  \node at (300:1) [label=right:$y$]{};
	  \node at (330:1) [label=right:$x$]{};
  \end{tikzpicture}
  &$\quad$&\begin{tikzpicture}[scale=1]
	  \path[draw,help lines] (0,0) circle (1 cm);
	  \tikzstyle{trait}=[dashed]
	  \draw[trait] (165:1.1) arc (165:215:1.1cm);
	  \draw[trait] (-35:1.1) arc (-35:15:1.1cm);
	  \tikzstyle{every node}=[shape=circle,fill=white,draw=black,minimum size=0.5pt,inner sep=1.5pt]
	  \node at (30:1)[label=right:$y$] {};
	  \node at (60:1)[label=right:$x$] {};
	  \node at (90:1) [label=above:$z$]{};
	  \node at (120:1)[label=left:$x$] {};
	  \node at (150:1)[label=left:$y$] {};
	  \node at (240:1) [label=left:$x$]{};
	  \node at (270:1) [label=below:$t$]{};
	  \node at (300:1) [label=right:$x$]{};
  \end{tikzpicture}
  \end{tabular}
  \end{center}
  \caption{Types of weighted cycles $\mathcal{C}_p$.}\label{figtypes}
\end{figure}

Note that the cycle in Figure~\ref{fig:cycle_0} is a particular case of all of these types. Such cycles are called \emph{Type $0$}. As we see in the next lemma, the case of Type $0$ cycles is easy to handle.

\begin{lemma}\label{prop_type1}
 For cycles $\mathcal{C}_p$ of Type $0$, represented by $zx^{p-1}$ with $1<p\in\mathbb{N}$, 
 all colourings are constant $2$-labellings.
\end{lemma}

\begin{proof}
  Let $ c$ be a colouring of $\mathcal{C}_p$.
  We set $\alpha_x$ to be the number of black vertices with weight $x$.
  If $ c(0)={\tt1}$, we have $\alpha_x+1$ black vertices and it is clear that $ c$ is a constant $2$-labelling
  where the weighted sum of black vertices is equal to $a=\alpha_x x +z$ (respectively $b=(\alpha_x+1)x$) if the vertex $0$ is black (resp. white).
\end{proof}

We now turn our attention to cycles of Type1mod which are not of Type 0.

\begin{lemma}\label{prop_type5}
 Let $p>2$ be an integer such that $p\equiv 1\pmod{4}$.
 For cycles $\mathcal{C}_p$ of Type1mod, represented by $z(xy)^\frac{p-1}{4}(yx)^\frac{p-1}{4}$, with $x\neq y$, there exists a non-trivial constant $2$-labelling $c$ if and only if $p\equiv0\pmod{3}$. In which case, $c$ is $3$-periodic of pattern period ${\tt 110}$.
\end{lemma}

\begin{proof}
 Let $p>2$ be an integer such that $p\equiv 1\pmod{4}$ and let $\mathcal{C}_p$ be a cycle of Type1mod with $x\neq y$. Assume that $c$ is a non-trivial constant $2$-labelling of $\mathcal{C}_p$. The colouring $c$ is not alternate since $p$ is odd. Hence, without loss of generality, we can assume that there exist two consecutive black vertices. Moreover, we can suppose that these vertices are the vertices $0$ and $1$ of $\mathcal{C}_p$. Indeed, $c$ is a constant $2$-labelling if and only if $c\circ \mathcal{R}_j$ is a constant labelling for all $j\in \mathbb{Z}$.
 
 For the colouring $c$, we let $\alpha_x$, $\alpha_y$ denote respectively the number of black vertices with weight $x$ and $y$. We have $a=\alpha_x x+\alpha_y y+z$. We consider the colour of the vertex $\frac{p+1}{2}$.
 
 Assume first that $c(\frac{p+1}{2})={\tt 1}$. Then, for the colouring $c\circ\mathcal{R}_1$, the sum of the weights of black vertices is
 $$a=(\alpha_x-1)y+z+(\alpha_y-1)x+y=\alpha_y x+\alpha_x y +z$$
 since under a $1$-rotation, any black vertex with weight $x$ becomes a black vertex of weight $y$, except for the vertex $1$ which becomes the vertex with weight $z$, and similarly any black vertex with weight $y$ becomes a black vertex of weight $x$ except for the vertex $\frac{p+1}{2}$ which becomes a vertex of weight $y$. As the weights $x$  and $y$ are distinct, it implies that $\alpha_x=\alpha_y$ (in order to have a sum of black vertices constant and equal to $a$). We set $\alpha:=\alpha_x=\alpha_y$ for a shorter notation.
 
 Let $i$ be the smallest integer in $\{0,\ldots,\frac{p-1}{2}-1\}$ such that $c(i+1)={\tt 0}$ and assume $c(\frac{p+1}{2}+\ell)={\tt 1}$ for any $\ell\in\{0,\ldots,i\}$ (otherwise, consider the colouring $c\circ\mathcal{R}_{\frac{p+1}{2}}$ instead of $c$). Then $c(\frac{p+1}{2}+i+1)={\tt 0}$ as depicted in Figure~\ref{fig:type1mod_case1_1}. With the colouring $c\circ\mathcal{R}_{i+1}$, we obtain a sum of the weights of black vertices equal to $b=\alpha x+(\alpha+1)y$ (Figure~\ref{fig:type1mod_case1_2}). To conclude this case, consider the vertex $i+2$ and observe that whatever value is assigned to $c(i+2)$, we obtain a contradiction (Figure~\ref{fig:type1mod_case1_3}).
 
 \begin{figure}[h!tbp]
 \begin{center}
 \scalebox{0.75}{
 \begin{tabular}{c c c}
  \begin{tikzpicture}[scale=1]
	  \path[draw,help lines] (0,0) circle (1 cm);
	  \tikzstyle{every node}=[shape=circle,fill=white,draw=black,minimum size=0.5pt,inner sep=0.5pt]
	  \tikzstyle{white}=[shape=circle,fill=white,draw=black,minimum size=0.5pt,inner sep=1.5pt]
	  \tikzstyle{black}=[shape=circle,fill=black,draw=black,minimum size=0.5pt,inner sep=1.5pt]
	  \tikzstyle{trait}=[dashed, draw=white,line width=0.75pt]
	  \draw[trait] (35:1) arc (35:85:1cm);
	  \draw[trait] (215:1) arc (215:255:1cm);
	  \node[white,label={[label distance=0.1cm]0:${i+1}$}] at (10:1){};	  
	  \node[black,label={[label distance=0.1cm]30:$i$}] at (30:1) {};
	  \node[black,label={[label distance=0.1cm]90:$0$}] at (90:1) {};
	  \node[black,label={[label distance=0.2cm]180:${\frac{p+1}{2}+i+1}$}] at (180:1) {};
	  \node[black,label={[label distance=0.1cm]180:${\frac{p+1}{2}+i}$}] at (200:1) {};
	  \node[black,label={[label distance=0.1cm]255:$\frac{p+1}{2}$}] at (255:1) {};
    \end{tikzpicture}
   &\begin{tikzpicture}
	  \node at (1,0){};
	  \node at (1,2) {$\mathcal{R}_i$};
   	  \draw[arrows={-stealth}, line width=1.5pt] (0,1) to (2,1);
    \end{tikzpicture}
   &\begin{tikzpicture}[scale=1]
  	  \path[draw,help lines] (0,0) circle (1 cm);
	  \tikzstyle{trait}=[dashed, draw=white,line width=0.75pt]
	  \tikzstyle{every node}=[shape=circle,fill=white,draw=black,minimum size=0.5pt,inner sep=0.5pt]
	  \tikzstyle{black}=[shape=circle,fill=black,draw=black,minimum size=0.5pt,inner sep=1.5pt]
	  \tikzstyle{white}=[shape=circle,fill=white,draw=black,minimum size=0.5pt,inner sep=1.5pt]
	  \draw[trait] (95:1) arc (95:145:1cm);
	  \draw[trait] (255:1) arc (255:325:1cm);
	  \node[white,label={[label distance=0.1cm]0:$1$}] at (70:1){};	  
	  \node[black,label={[label distance=0.1cm]90:$0$}] at (90:1) {};
	  \node[black,label={[label distance=0.1cm]180:${p-i}$}] at (150:1){};
	  \node[black,label={[label distance=0.1cm]180:$\frac{p+1}{2}+1$}] at (235:1){};
  	  \node[black,label={[label distance=0.1cm]255:$\frac{p+1}{2}$}] at (255:1) {};
	  \node[black,label={[label distance=0.1cm]0:${\frac{p+1}{2}-i}$}] at (330:1) {};
    \end{tikzpicture}\\
    $a=\alpha x+\alpha y+z$ & & $a=\alpha x+\alpha y+z$\\ & &\\
    & & \begin{tikzpicture}
	  \node at (-1.5,0){};	
	  \node at (1,0.5) {$\mathcal{R}_\frac{p+1}{2}$};
   	  \draw[arrows={-stealth}, line width=1.5pt] (0,1) to (0,0);
    \end{tikzpicture}\\
	\begin{tikzpicture}[scale=1]
  	  \path[draw,help lines] (0,0) circle (1 cm);
	  \tikzstyle{trait}=[dashed, draw=white,line width=0.75pt]
	  \tikzstyle{every node}=[shape=circle,fill=white,draw=black,minimum size=0.5pt,inner sep=0.5pt]
	  \tikzstyle{black}=[shape=circle,fill=black,draw=black,minimum size=0.5pt,inner sep=1.5pt]
	  \tikzstyle{white}=[shape=circle,fill=white,draw=black,minimum size=0.5pt,inner sep=1.5pt]
	  \draw[trait] (110:1) arc (110:170:1cm);
	  \draw[trait] (305:1) arc (305:370:1cm);  
	  \node[black,label={[label distance=0.1cm]90:$0$}] at (90:1) {};
  	  \node[black,label={[label distance=0.1cm]180:$p-1$}] at (110:1){};	
	  \node[black,label={[label distance=0.1cm]180:${p-i-1}$}] at (170:1){};
	  \node[white,label={[label distance=0.1cm]285:$\frac{p-1}{2}$}]  at (285:1) {};
  	  \node[black,label={[label distance=0.1cm]0:$\frac{p-1}{2}-1$}] at (305:1) {};
	  \node[black,label={[label distance=0.1cm]0:${\frac{p+1}{2}-i-2}$}] at (370:1) {};
    \end{tikzpicture}  &
    \begin{tikzpicture}
	  \node at (1,0){};
	  \node at (1,2) {$\mathcal{R}_1$};
   	  \draw[arrows={-stealth}, line width=1.5pt] (2,1) to (0,1);
    \end{tikzpicture}
    &  \begin{tikzpicture}[scale=1]
  	  \path[draw,help lines] (0,0) circle (1 cm);
	  \tikzstyle{trait}=[dashed, draw=white,line width=0.75pt]
	  \tikzstyle{every node}=[shape=circle,fill=white,draw=black,minimum size=0.5pt,inner sep=0.5pt]
	  \tikzstyle{black}=[shape=circle,fill=black,draw=black,minimum size=0.5pt,inner sep=1.5pt]
	  \tikzstyle{white}=[shape=circle,fill=white,draw=black,minimum size=0.5pt,inner sep=1.5pt]
	  \draw[trait] (95:1) arc (95:145:1cm);
	  \draw[trait] (285:1) arc (285:350:1cm);
	  \node[black,label={[label distance=0.1cm]0:$1$}] at (70:1){};	  
	  \node[black,label={[label distance=0.1cm]90:$0$}] at (90:1) {};
	  \node[black,label={[label distance=0.1cm]180:${p-i}$}] at (150:1){};
  	  \node[white,label={[label distance=0.1cm]255:$\frac{p+1}{2}$}] at (255:1) {};
	  \node[black,label={[label distance=0.1cm]285:$\frac{p-1}{2}$}]  at (285:1) {};
	  \node[black,label={[label distance=0.1cm]0:${\frac{p+1}{2}-i-1}$}] at (350:1) {};
    \end{tikzpicture}  \\
       $a=(\alpha-1)y+z+\alpha x+x$&&$a=\alpha x+\alpha y+z$ 
 \end{tabular}}
 \end{center}
 \caption{Rotations of the colouring $c$ of a Type1mod cycle with $c(\frac{p+1}{2}+i+1)={\tt 1}$, and their corresponding weighted sums of black vertices which are not all equal.}\label{fig:type1mod_case1_1}
 \end{figure}
 
\begin{figure}[h!tbp]
 \begin{center}
 \scalebox{0.74}{
 \begin{tabular}{c c c }
  \begin{tikzpicture}[scale=1]
	  \path[draw,help lines] (0,0) circle (1 cm);
	  \tikzstyle{every node}=[shape=circle,fill=white,draw=black,minimum size=0.5pt,inner sep=0.5pt]
	  \tikzstyle{white}=[shape=circle,fill=white,draw=black,minimum size=0.5pt,inner sep=1.5pt]
	  \tikzstyle{black}=[shape=circle,fill=black,draw=black,minimum size=0.5pt,inner sep=1.5pt]
	  \tikzstyle{trait}=[dashed, draw=white,line width=0.75pt]
	  \draw[trait] (35:1) arc (35:85:1cm);
	  \draw[trait] (200:1) arc (200:255:1cm);
	  \node[white,label={[label distance=0.1cm]0:${i+1}$}] at (10:1){};	  
	  \node[black,label={[label distance=0.1cm]30:$i$}] at (30:1) {};
	  \node[black,label={[label distance=0.1cm]90:$0$}] at (90:1) {};
	  \node[white,label={[label distance=0.2cm]180:${\frac{p+1}{2}+i+1}$}] at (180:1) {};
	  \node[black,label={[label distance=0.1cm]180:${\frac{p+1}{2}+i}$}] at (200:1) {};
	  \node[black,label={[label distance=0.1cm]255:$\frac{p+1}{2}$}] at (255:1) {};
    \end{tikzpicture}
   &\begin{tikzpicture}[scale=1]
         \begin{scope}[yshift=-1.3cm]
      \node at (-2.8,0.5) {$\mathcal{R}_{i}$};
   	  \draw[arrows={-stealth}, line width=1.5pt] (-3.3,0) to (-2.3,0);
   	  \begin{scope}[xshift=4.5cm]
   	  \node at (-2,0.5) {$\mathcal{R}_1$};
   	  \draw[arrows={-stealth}, line width=1.5pt] (-2.4,0) to (-1.4,0);
   	  \end{scope}\end{scope}
  	  \path[draw,help lines] (0,0) circle (1 cm);
	  \tikzstyle{trait}=[dashed, draw=white,line width=0.75pt]
	  \tikzstyle{every node}=[shape=circle,fill=white,draw=black,minimum size=0.5pt,inner sep=0.5pt]
	  \tikzstyle{black}=[shape=circle,fill=black,draw=black,minimum size=0.5pt,inner sep=1.5pt]
	  \tikzstyle{white}=[shape=circle,fill=white,draw=black,minimum size=0.5pt,inner sep=1.5pt]
	  \draw[trait] (95:1) arc (95:145:1cm);
	  \draw[trait] (255:1) arc (255:325:1cm);
	  \node[white,label={[label distance=0.1cm]0:$1$}] at (70:1){};	  
	  \node[black,label={[label distance=0.1cm]90:$0$}] at (90:1) {};
	  \node[black,label={[label distance=0.1cm]180:${p-i}$}] at (150:1){};
	  \node[white,label={[label distance=0.1cm]180:$\frac{p+1}{2}+1$}] at (235:1){};
  	  \node[black,label={[label distance=0.1cm]255:$\frac{p+1}{2}$}] at (255:1) {};
	  \node[black,label={[label distance=0.1cm]0:${\frac{p+1}{2}-i}$}] at (330:1) {};
    \end{tikzpicture}
    &
    \begin{tikzpicture}[scale=1]
  	  \path[draw,help lines] (0,0) circle (1 cm);
	  \tikzstyle{trait}=[dashed, draw=white,line width=0.75pt]
	  \tikzstyle{every node}=[shape=circle,fill=white,draw=black,minimum size=0.5pt,inner sep=0.5pt]
	  \tikzstyle{black}=[shape=circle,fill=black,draw=black,minimum size=0.5pt,inner sep=1.5pt]
	  \tikzstyle{white}=[shape=circle,fill=white,draw=black,minimum size=0.5pt,inner sep=1.5pt]
	  \draw[trait] (110:1) arc (110:170:1cm);
	  \draw[trait] (285:1) arc (285:350:1cm); 
	  \node[white,label={[label distance=0.1cm]90:$0$}] at (90:1) {};
	  \node[black,label={[label distance=0.1cm]180:$p-1$}] at (110:1){};	 
	  \node[black,label={[label distance=0.1cm]180:${p-i-1}$}] at (170:1){};
  	  \node[white,label={[label distance=0.1cm]255:$\frac{p+1}{2}$}] at (255:1) {};
	  \node[black,label={[label distance=0.1cm]285:$\frac{p-1}{2}$}]  at (285:1) {};
	  \node[black,label={[label distance=0.1cm]0:${\frac{p+1}{2}-i-1}$}] at (350:1) {};
    \end{tikzpicture}
    \\
    $a=\alpha x+\alpha y+z$ & $a=\alpha x+\alpha y+z$ & $b=\alpha y+(\alpha-1)x+y+x$ 
 \end{tabular}}
 \end{center}
 \caption{Rotations of the colouring $c$ of a Type1mod cycle with $c(\frac{p+1}{2}+i+1)={\tt 0}$, and their corresponding weighted sums of black vertices.}\label{fig:type1mod_case1_2}
 \end{figure} 

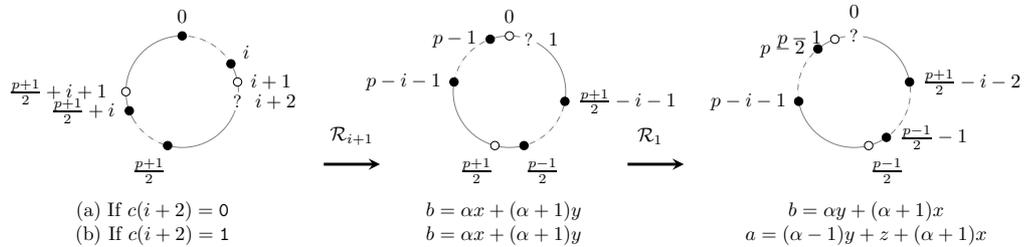
\begin{figure}[h!tbp]
\begin{center}
\scalebox{0.74}{
	\begin{tabular}{ccc}
	 \begin{tikzpicture}[scale=1]
	  \path[draw,help lines] (0,0) circle (1 cm);
	  \tikzstyle{every node}=[shape=circle,fill=white,draw=black,minimum size=0.5pt,inner sep=0.5pt]
	  \tikzstyle{white}=[shape=circle,fill=white,draw=black,minimum size=0.5pt,inner sep=1.5pt]
	  \tikzstyle{black}=[shape=circle,fill=black,draw=black,minimum size=0.5pt,inner sep=1.5pt]
	  \tikzstyle{trait}=[dashed, draw=white,line width=0.75pt]
	  \draw[trait] (35:1) arc (35:85:1cm);
	  \draw[trait] (215:1) arc (215:255:1cm);
	  \node[draw=white,fill=white,label={[label distance=0.1cm]0:${i+2}$}] at (350:1){\small{?}};	
	  \node[white,label={[label distance=0.1cm]0:${i+1}$}] at (10:1){};	  
	  \node[black,label={[label distance=0.1cm]30:$i$}] at (30:1) {};
	  \node[black,label={[label distance=0.1cm]90:$0$}] at (90:1) {};
	  \node[white,label={[label distance=0.2cm]180:${\frac{p+1}{2}+i+1}$}] at (180:1) {};
	  \node[black,label={[label distance=0.1cm]180:${\frac{p+1}{2}+i}$}] at (200:1) {};
	  \node[black,label={[label distance=0.1cm]255:$\frac{p+1}{2}$}] at (255:1) {};
    \end{tikzpicture}
   &
    \begin{tikzpicture}[scale=1]
      \begin{scope}[yshift=-1.3cm]
      \node at (-2.8,0.5) {$\mathcal{R}_{i+1}$};
   	  \draw[arrows={-stealth}, line width=1.5pt] (-3.3,0) to (-2.3,0);
   	  \begin{scope}[xshift=4.5cm]
   	  \node at (-2,0.5) {$\mathcal{R}_1$};
   	  \draw[arrows={-stealth}, line width=1.5pt] (-2.4,0) to (-1.4,0);
   	  \end{scope}\end{scope}
   	  
  	  \path[draw,help lines] (0,0) circle (1 cm);
	  \tikzstyle{trait}=[dashed, draw=white,line width=0.75pt]
	  \tikzstyle{every node}=[shape=circle,fill=white,draw=black,minimum size=0.5pt,inner sep=0.5pt]
	  \tikzstyle{black}=[shape=circle,fill=black,draw=black,minimum size=0.5pt,inner sep=1.5pt]
	  \tikzstyle{white}=[shape=circle,fill=white,draw=black,minimum size=0.5pt,inner sep=1.5pt]
	  \draw[trait] (110:1) arc (110:170:1cm);
	  \draw[trait] (285:1) arc (285:350:1cm); 
	  \node[draw=none,fill=white,label={[label distance=0.1cm]0:$1$}] at (70:1){\small{?}};	
	  \node[white,label={[label distance=0.1cm]90:$0$}] at (90:1) {};
	  \node[black,label={[label distance=0.1cm]180:$p-1$}] at (110:1){};	 
	  \node[black,label={[label distance=0.1cm]180:${p-i-1}$}] at (170:1){};
  	  \node[white,label={[label distance=0.1cm]255:$\frac{p+1}{2}$}] at (255:1) {};
	  \node[black,label={[label distance=0.1cm]285:$\frac{p-1}{2}$}]  at (285:1) {};
	  \node[black,label={[label distance=0.1cm]0:${\frac{p+1}{2}-i-1}$}] at (350:1) {};
    \end{tikzpicture}
    &\begin{tikzpicture}[scale=1]
  	  \path[draw,help lines] (0,0) circle (1 cm);
	  \tikzstyle{trait}=[dashed, draw=white,line width=0.75pt]
	  \tikzstyle{every node}=[shape=circle,fill=white,draw=black,minimum size=0.5pt,inner sep=0.5pt]
	  \tikzstyle{black}=[shape=circle,fill=black,draw=black,minimum size=0.5pt,inner sep=1.5pt]
	  \tikzstyle{white}=[shape=circle,fill=white,draw=black,minimum size=0.5pt,inner sep=1.5pt]
	  \draw[trait] (130:1) arc (130:190:1cm);
	  \draw[trait] (305:1) arc (305:370:1cm); 
	  \node[draw=none,fill=white,label={[label distance=0.1cm]90:$0$}] at (90:1) {\small{?}};
	  \node[white,label={[label distance=0.1cm]180:$p-1$}] at (110:1){};	 
	  \node[black,label={[label distance=0.1cm]180:$p-2$}] at (130:1){};	
	  \node[black,label={[label distance=0.1cm]180:${p-i-1}$}] at (190:1){};
  	  \node[white,label={[label distance=0.1cm]285:$\frac{p-1}{2}$}] at (285:1) {};
	  \node[black,label={[label distance=0.1cm]0:$\frac{p-1}{2}-1$}]  at (305:1) {};
	  \node[black,label={[label distance=0.1cm]0:${\frac{p+1}{2}-i-2}$}] at (370:1) {};
    \end{tikzpicture}\\
    (a) If $c(i+2)={\tt 0}$&$b=\alpha x+(\alpha+1)y$ & $b=\alpha y+ (\alpha+1)x$\\
    (b) If $c(i+2)={\tt 1}$&$b=\alpha x+(\alpha+1)y$ & $a=(\alpha -1)y+z+(\alpha+1)x$
	\end{tabular}
}
\end{center}
\caption{Rotations of the colouring $c$ of a Type1mod cycle and their corresponding weighted sums of black vertices depending on the colour $c(i+2)$.}\label{fig:type1mod_case1_3}
\end{figure}
 
 \smallskip
Therefore, we have $c(\frac{p+1}{2})={\tt 0}$ and $a=\alpha_x x+\alpha_y y+z$ as in the beginning. Observe that the previous reasoning means that for any integer $j$, we have
\begin{equation}\label{eq:type1mod}
c\circ\mathcal{R}_j(0)={\tt 1}=c\circ\mathcal{R}_j(1)\Rightarrow c\circ\mathcal{R}_j\left(\frac{p+1}{2}\right)={\tt 0}.
\end{equation}
With the colouring $c\circ\mathcal{R}_1$, the sum of the weights of black vertices is
$$a=(\alpha_x-1)y+z+\alpha_y x+x=(\alpha_y+1)x+(\alpha_x-1)y+z.$$
Since $x\ne y$, we get $\alpha_x=\alpha_y+1$. We set $\alpha:=\alpha_y$ for a shorter notation.
 
 Let $i$ be the smallest integer in $\{0,\ldots,\frac{p-1}{2}-1\}$ such that $c(i+1)={\tt 0}$. From Equation~\eqref{eq:type1mod}, we have $c(\frac{p+1}{2}+\ell)={\tt 0}$ for any $\ell\in\{0,\ldots,i-1\}$. Moreover, we have $c(\frac{p+1}{2}+i)={\tt 1}$. Indeed, assume that $c(\frac{p+1}{2}+i)={\tt 0}$ (Figure~\ref{fig:type1mod_case2_1}), then with the colouring $c\circ\mathcal{R}_{i+1}$ we obtain a sum of the weights of black vertices equal to $b=(\alpha+1)x+(\alpha+1)y$. As $c$ is a constant $2$-labelling, with the colouring $c\circ\mathcal{R}_{\frac{p+1}{2}}$, we have the same weighted sum $b$. Then it implies that the weighted sum $b$ with the colouring $c\circ\mathcal{R}_{\frac{p+1}{2}+1}$ has a different value, which is a contradiction. So  $c(\frac{p+1}{2}+i)={\tt 1}$ and with the colouring $c\circ\mathcal{R}_{i+1}$, we have a sum of the weights of black vertices equal to $b=\alpha x+(\alpha+2)y$ (Figure~\ref{fig:type1mod_case2_2}).

\begin{figure}[h!tbp]
\begin{center}
\scalebox{0.74}{
\begin{tabular}{ccccc}
 \begin{tikzpicture}[scale=1]
	 \node at (0,-2.3){$a=(\alpha+1)x+\alpha y+z$};\node at (0,-5){};
	  \path[draw,help lines] (0,0) circle (1 cm);
	  \tikzstyle{every node}=[shape=circle,fill=white,draw=black,minimum size=0.5pt,inner sep=0.5pt]
	  \tikzstyle{white}=[shape=circle,fill=white,draw=black,minimum size=0.5pt,inner sep=1.5pt]
	  \tikzstyle{black}=[shape=circle,fill=black,draw=black,minimum size=0.5pt,inner sep=1.5pt]
	  \tikzstyle{trait}=[dashed, draw=white,line width=0.75pt]
	  \draw[trait] (35:1) arc (35:85:1cm);
	  \draw[trait] (215:1) arc (215:255:1cm);
	  \node[white,label={[label distance=0.1cm]0:${i+1}$}] at (10:1){};	  
	  \node[black,label={[label distance=0.1cm]30:$i$}] at (30:1) {};
	  \node[black,label={[label distance=0.1cm]90:$0$}] at (90:1) {};
	  \node[white,label={[label distance=0.1cm]180:${\frac{p+1}{2}+i}$}] at (190:1) {};
	  \node[white,label={[label distance=0.2cm]180:${\frac{p+1}{2}+i-1}$}] at (210:1) {};
	  \node[white,label={[label distance=0.1cm]255:$\frac{p+1}{2}$}] at (255:1) {};
    \end{tikzpicture}
   &\begin{tikzpicture}[scale=1]
	  \node at (0,-2){$a=(\alpha+1)x+\alpha y+z$};
	  \node at (0,-6.5){$b=(\alpha+1)x+(\alpha+1)y$};
	  \node at (-2.5,-0.8) {$\mathcal{R}_i$};
   	  \draw[arrows={-stealth}, line width=1.5pt] (-2.5,-1.5) to (-1.5,-0.5);
	  \node at (-2.5,-3.7) {$\mathcal{R}_{\frac{p+1}{2}}$};
   	  \draw[arrows={-stealth}, line width=1.5pt] (-2.5,-3) to (-1.5,-4);   
   	  \begin{scope}[xshift=4.2cm]
   	  \node at (-2,0.5) {$\mathcal{R}_1$};
   	  \draw[arrows={-stealth}, line width=1.5pt] (-2.4,0) to (-1.4,0);
	  \node at (-2,-4) {$\mathcal{R}_1$};
   	  \draw[arrows={-stealth}, line width=1.5pt] (-2.4,-4.5) to (-1.4,-4.5);
   	  \end{scope}

   	  \path[draw,help lines] (0,0) circle (1 cm);
	  \tikzstyle{trait}=[dashed, draw=white,line width=0.75pt]
	  \tikzstyle{every node}=[shape=circle,fill=white,draw=black,minimum size=0.5pt,inner sep=0.5pt]
	  \tikzstyle{black}=[shape=circle,fill=black,draw=black,minimum size=0.5pt,inner sep=1.5pt]
	  \tikzstyle{white}=[shape=circle,fill=white,draw=black,minimum size=0.5pt,inner sep=1.5pt]
	  \draw[trait] (95:1) arc (95:145:1cm);
	  \draw[trait] (285:1) arc (285:325:1cm);
	  \node[white,label={[label distance=0.1cm]0:$1$}] at (70:1){};	  
	  \node[black,label={[label distance=0.1cm]90:$0$}] at (90:1) {};
	  \node[black,label={[label distance=0.1cm]180:${p-i}$}] at (150:1){};
	  \node[white,label={[label distance=0.1cm]255:$\frac{p+1}{2}$}] at (255:1) {};
	  \node[white,label={[label distance=0.1cm]285:$\frac{p-1}{2}$}]  at (285:1) {};
	  \node[white,label={[label distance=0.1cm]0:${\frac{p+1}{2}-i}$}] at (330:1) {};
	  
	  \begin{scope}[yshift=-4.5cm]
	  \path[draw,help lines] (0,0) circle (1 cm);
	  \tikzstyle{every node}=[shape=circle,fill=white,draw=black,minimum size=0.5pt,inner sep=0.5pt]
	  \tikzstyle{white}=[shape=circle,fill=white,draw=black,minimum size=0.5pt,inner sep=1.5pt]
	  \tikzstyle{black}=[shape=circle,fill=black,draw=black,minimum size=0.5pt,inner sep=1.5pt]
	  \tikzstyle{trait}=[dashed, draw=white,line width=0.75pt]
	  \draw[trait] (50:1) arc (50:85:1cm);
	  \draw[trait] (215:1) arc (215:280:1cm);
	  \node[white,label={[label distance=0.1cm]30:$i$}] at (30:1) {};
	  \node[white,label={[label distance=-0.2cm]90:${i-1}$}] at (50:1){};	  
	  \node[white,label={[label distance=0.1cm]90:$0$}] at (90:1) {};
	  \node[white,label={[label distance=0.1cm]180:${\frac{p+1}{2}+i}$}] at (190:1) {};
	  \node[black,label={[label distance=0.2cm]180:${\frac{p+1}{2}+i-1}$}] at (210:1) {};
	  \node[black,label={[label distance=0.1cm]285:$\frac{p-1}{2}$}]  at (285:1) {};
	  \end{scope}
    \end{tikzpicture}
    &\begin{tikzpicture}[scale=1]
      \node at (0,-2){$b=(\alpha+1)y+(\alpha+1)x$};
	  \node at (0,-6.5){$b=(\alpha+1)y+\alpha x+y$};
  	  \path[draw,help lines] (0,0) circle (1 cm);
	  \tikzstyle{trait}=[dashed, draw=white,line width=0.75pt]
	  \tikzstyle{every node}=[shape=circle,fill=white,draw=black,minimum size=0.5pt,inner sep=0.5pt]
	  \tikzstyle{black}=[shape=circle,fill=black,draw=black,minimum size=0.5pt,inner sep=1.5pt]
	  \tikzstyle{white}=[shape=circle,fill=white,draw=black,minimum size=0.5pt,inner sep=1.5pt]
	  \draw[trait] (110:1) arc (110:170:1cm);
	  \draw[trait] (305:1) arc (305:350:1cm); 
	  \node[white,label={[label distance=0.1cm]90:$0$}] at (90:1) {};
	  \node[black,label={[label distance=0.1cm]180:$p-1$}] at (110:1){};	 
	  \node[black,label={[label distance=0.1cm]180:${p-i-1}$}] at (170:1){};
	  \node[white,label={[label distance=0.1cm]285:$\frac{p-1}{2}$}]  at (285:1) {};
  	  \node[white,label={[label distance=0.1cm]0:$\frac{p-1}{2}-1$}] at (305:1) {};
	  \node[white,label={[label distance=0.1cm]0:${\frac{p+1}{2}-i-1}$}] at (350:1) {};
	  
	  \begin{scope}[yshift=-4.5cm]
	  \path[draw,help lines] (0,0) circle (1 cm);
	  \tikzstyle{every node}=[shape=circle,fill=white,draw=black,minimum size=0.5pt,inner sep=0.5pt]
	  \tikzstyle{white}=[shape=circle,fill=white,draw=black,minimum size=0.5pt,inner sep=1.5pt]
	  \tikzstyle{black}=[shape=circle,fill=black,draw=black,minimum size=0.5pt,inner sep=1.5pt]
	  \tikzstyle{trait}=[dashed, draw=white,line width=0.75pt]
	  \draw[trait] (50:1) arc (50:85:1cm);
	  \draw[trait] (240:1) arc (240:305:1cm);
	  \node[white,label={[label distance=0.1cm]0:${i-1}$}] at (50:1){};	  
	  \node[white,label={[label distance=0.1cm]90:$0$}] at (90:1) {};
	  \node[white,label={[label distance=0.1cm]180:$p-1$}] at (110:1){};
	  \node[white,label={[label distance=0.2cm]180:${\frac{p+1}{2}+i-1}$}] at (210:1) {};
	  \node[black,label={[label distance=0.1cm]180:${\frac{p+1}{2}+i-2}$}] at (240:1) {};
	  \node[black,label={[label distance=0.1cm]0:$\frac{p-1}{2}+1$}]  at (305:1) {};
	  \end{scope}
    \end{tikzpicture}
    \end{tabular}
}
\end{center}
\caption{Rotations of the colouring $c$ of a Type1mod cycle with $c(\frac{p+1}{2}+i)={\tt 0}$, and their corresponding weighted sums $b$ of black vertices which are not all equal.}\label{fig:type1mod_case2_1}
\end{figure}

\begin{figure}[h!tbp]
\begin{center}
\scalebox{0.75}{
\begin{tabular}{ccc}
 \begin{tikzpicture}[scale=1]
	  \path[draw,help lines] (0,0) circle (1 cm);
	  \tikzstyle{every node}=[shape=circle,fill=white,draw=black,minimum size=0.5pt,inner sep=0.5pt]
	  \tikzstyle{white}=[shape=circle,fill=white,draw=black,minimum size=0.5pt,inner sep=1.5pt]
	  \tikzstyle{black}=[shape=circle,fill=black,draw=black,minimum size=0.5pt,inner sep=1.5pt]
	  \tikzstyle{trait}=[dashed, draw=white,line width=0.75pt]
	  \draw[trait] (35:1) arc (35:85:1cm);
	  \draw[trait] (215:1) arc (215:255:1cm);
	  \node[white,label={[label distance=0.1cm]0:${i+1}$}] at (10:1){};	  
	  \node[black,label={[label distance=0.1cm]30:$i$}] at (30:1) {};
	  \node[black,label={[label distance=0.1cm]90:$0$}] at (90:1) {};
	  \node[black,label={[label distance=0.1cm]180:${\frac{p+1}{2}+i}$}] at (190:1) {};
	  \node[white,label={[label distance=0.2cm]180:${\frac{p+1}{2}+i-1}$}] at (210:1) {};
	  \node[white,label={[label distance=0.1cm]255:$\frac{p+1}{2}$}] at (255:1) {};
    \end{tikzpicture}
    &
    \begin{tikzpicture}[scale=1]
	  \node at (-2.5,0.5) {$\mathcal{R}_i$};
   	  \draw[arrows={-stealth}, line width=1.5pt] (-3,0) to (-2,0);
   	  \begin{scope}[xshift=4.2cm]
   	  \node at (-2,0.5) {$\mathcal{R}_1$};
   	  \draw[arrows={-stealth}, line width=1.5pt] (-2.4,0) to (-1.4,0);
   	  \end{scope}

   	  \path[draw,help lines] (0,0) circle (1 cm);
	  \tikzstyle{trait}=[dashed, draw=white,line width=0.75pt]
	  \tikzstyle{every node}=[shape=circle,fill=white,draw=black,minimum size=0.5pt,inner sep=0.5pt]
	  \tikzstyle{black}=[shape=circle,fill=black,draw=black,minimum size=0.5pt,inner sep=1.5pt]
	  \tikzstyle{white}=[shape=circle,fill=white,draw=black,minimum size=0.5pt,inner sep=1.5pt]
	  \draw[trait] (95:1) arc (95:145:1cm);
	  \draw[trait] (285:1) arc (285:325:1cm);
	  \node[white,label={[label distance=0.1cm]0:$1$}] at (70:1){};	  
	  \node[black,label={[label distance=0.1cm]90:$0$}] at (90:1) {};
	  \node[black,label={[label distance=0.1cm]180:${p-i}$}] at (150:1){};
	  \node[black,label={[label distance=0.1cm]255:$\frac{p+1}{2}$}] at (255:1) {};
	  \node[white,label={[label distance=0.1cm]285:$\frac{p-1}{2}$}]  at (285:1) {};
	  \node[white,label={[label distance=0.1cm]0:${\frac{p+1}{2}-i}$}] at (330:1) {};
	  \end{tikzpicture}
	  &
	  \begin{tikzpicture}[scale=1]
  	  \path[draw,help lines] (0,0) circle (1 cm);
	  \tikzstyle{trait}=[dashed, draw=white,line width=0.75pt]
	  \tikzstyle{every node}=[shape=circle,fill=white,draw=black,minimum size=0.5pt,inner sep=0.5pt]
	  \tikzstyle{black}=[shape=circle,fill=black,draw=black,minimum size=0.5pt,inner sep=1.5pt]
	  \tikzstyle{white}=[shape=circle,fill=white,draw=black,minimum size=0.5pt,inner sep=1.5pt]
	  \draw[trait] (110:1) arc (110:170:1cm);
	  \draw[trait] (305:1) arc (305:350:1cm); 
	  \node[white,label={[label distance=0.1cm]90:$0$}] at (90:1) {};
	  \node[black,label={[label distance=0.1cm]180:$p-1$}] at (110:1){};	 
	  \node[black,label={[label distance=0.1cm]180:${p-i-1}$}] at (170:1){};
	  \node[black,label={[label distance=0.1cm]285:$\frac{p-1}{2}$}]  at (285:1) {};
  	  \node[white,label={[label distance=0.1cm]0:$\frac{p-1}{2}-1$}] at (305:1) {};
	  \node[white,label={[label distance=0.1cm]0:${\frac{p+1}{2}-i-1}$}] at (350:1) {};
	  \end{tikzpicture}	  \\
	  $a=(\alpha+1)x+\alpha y+z$ & $a=(\alpha+1)x+\alpha y+z$ &  $b=(\alpha+1)y+(\alpha-1)x+y+x$
\end{tabular}
}
\end{center}
\caption{Rotations of the colouring $c$ of a Type1mod cycle with $c(\frac{p+1}{2}+i)={\tt1}$, and their corresponding weighted sums of black vertices.}\label{fig:type1mod_case2_2}
\end{figure}

From $b=\alpha x+(\alpha+2)y$, it follows that $i$ must be equal to $2$, otherwise the colouring $c\circ\mathcal{R}_{\frac{p+1}{2}}$ leads to a different sum of the weights of black vertices (Figure~\ref{fig:type1mod_case2_3}).
Then we have $c(3)={\tt 1}$ (Figure~\ref{fig:type1mod_case2_4}). Similarly $c(\frac{p+1}{2}+2)={\tt 1}$ (Figure~\ref{fig:type1mod_case2_5}).

\begin{figure}[h!tbp]
\begin{center}
\scalebox{0.75}{
\begin{tabular}{ccc}
 	\begin{tikzpicture}[scale=1]
	  \path[draw,help lines] (0,0) circle (1 cm);
	  \tikzstyle{every node}=[shape=circle,fill=white,draw=black,minimum size=0.5pt,inner sep=0.5pt]
	  \tikzstyle{white}=[shape=circle,fill=white,draw=black,minimum size=0.5pt,inner sep=1.5pt]
	  \tikzstyle{black}=[shape=circle,fill=black,draw=black,minimum size=0.5pt,inner sep=1.5pt]
	  \tikzstyle{trait}=[dashed, draw=white,line width=0.75pt]
	  \draw[trait] (35:1) arc (35:85:1cm);
	  \draw[trait] (215:1) arc (215:255:1cm);
	  \node[white,label={[label distance=0.1cm]0:${i+1}$}] at (10:1){};	  
	  \node[black,label={[label distance=0.1cm]30:$i$}] at (30:1) {};
	  \node[black,label={[label distance=0.1cm]90:$0$}] at (90:1) {};
	  \node[black,label={[label distance=0.1cm]90:$1$}] at (70:1) {};
	  \node[black,label={[label distance=0.1cm]180:${\frac{p+1}{2}+i}$}] at (190:1) {};
	  \node[white,label={[label distance=0.2cm]180:${\frac{p+1}{2}+i-1}$}] at (210:1) {};
	  \node[white,label={[label distance=0.1cm]255:$\frac{p+1}{2}$}] at (255:1) {};
    \end{tikzpicture}
    &     \begin{tikzpicture}[scale=1]
	  \node at (-2.5,0.7) {$\mathcal{R}_{\frac{p+1}{2}}$};
   	  \draw[arrows={-stealth}, line width=1.5pt] (-3,0.2) to (-2,0.2);
   	  \begin{scope}[xshift=4.2cm]
   	  \node at (-2,0.7) {$\mathcal{R}_1$};
   	  \draw[arrows={-stealth}, line width=1.5pt] (-2.4,0.2) to (-1.4,0.2);
   	  \end{scope}
   	  \path[draw,help lines] (0,0) circle (1 cm);
	  \tikzstyle{every node}=[shape=circle,fill=white,draw=black,minimum size=0.5pt,inner sep=0.5pt]
	  \tikzstyle{white}=[shape=circle,fill=white,draw=black,minimum size=0.5pt,inner sep=1.5pt]
	  \tikzstyle{black}=[shape=circle,fill=black,draw=black,minimum size=0.5pt,inner sep=1.5pt]
	  \tikzstyle{trait}=[dashed, draw=white,line width=0.75pt]
	  \draw[trait] (50:1) arc (50:85:1cm);
	  \draw[trait] (215:1) arc (215:255:1cm);
	  \node[black,label={[label distance=0.1cm]30:$i$}] at (30:1) {};
	  \node[white,label={[label distance=-0.2cm]90:${i-1}$}] at (50:1){};	  
	  \node[white,label={[label distance=0.1cm]90:$0$}] at (90:1) {};
	  \node[white,label={[label distance=0.1cm]180:${\frac{p+1}{2}+i}$}] at (190:1) {};
	  \node[black,label={[label distance=0.2cm]180:${\frac{p+1}{2}+i-1}$}] at (210:1) {};
	  \node[black,label={[label distance=0.1cm]255:$\frac{p+1}{2}$}] at (255:1) {};
	  \node[black,label={[label distance=0.1cm]285:$\frac{p-1}{2}$}]  at (285:1) {};
	  \end{tikzpicture}
	  & \begin{tikzpicture}[scale=1]
	  \path[draw,help lines] (0,0) circle (1 cm);
	  \tikzstyle{every node}=[shape=circle,fill=white,draw=black,minimum size=0.5pt,inner sep=0.5pt]
	  \tikzstyle{white}=[shape=circle,fill=white,draw=black,minimum size=0.5pt,inner sep=1.5pt]
	  \tikzstyle{black}=[shape=circle,fill=black,draw=black,minimum size=0.5pt,inner sep=1.5pt]
	  \tikzstyle{trait}=[dashed, draw=white,line width=0.75pt]
	  \draw[trait] (60:1) arc (60:85:1cm);
	  \draw[trait] (235:1) arc (235:285:1cm);
	  \node[black] at (60:1){};	  
	  \node[black,label={[label distance=0.1cm]0:${i-1}$}] at (40:1){};	  
	  \node[white,label={[label distance=0.1cm]90:$0$}] at (90:1) {};
	  \node[white,label={[label distance=0.1cm]180:$p-1$}] at (110:1){};
	  \node[white,label={[label distance=0.2cm]180:${\frac{p+1}{2}+i-1}$}] at (210:1) {};
	  \node[black,label={[label distance=0.1cm]180:${\frac{p+1}{2}+i-2}$}] at (230:1) {};
	  \node[black,label={[label distance=0.1cm]285:$\frac{p-1}{2}$}] at (285:1) {};
	  \node[black,label={[label distance=0.1cm]0:$\frac{p-1}{2}+1$}]  at (305:1) {};
    \end{tikzpicture}\\
    $a=(\alpha+1)x+\alpha y+z$ & $b=\alpha x+(\alpha+2) y$ & $b=\alpha y+(\alpha+1)x+y$
\end{tabular}}
\end{center}
\caption{Rotations of the colouring $c$ of a Type1mod cycle with $c(j)={\tt 1}$ for all $0\le j\le i$ with $i>1$, and their corresponding weighted sums of black vertices distinct which are not all equal.}\label{fig:type1mod_case2_3}
\end{figure}
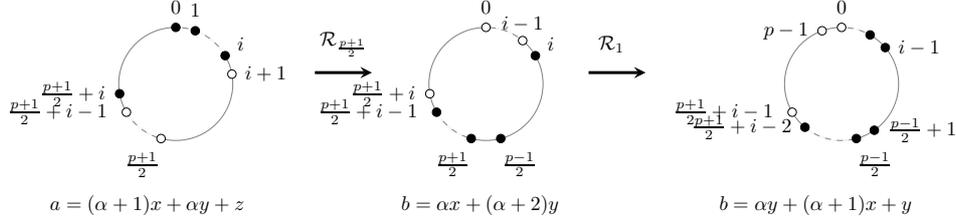

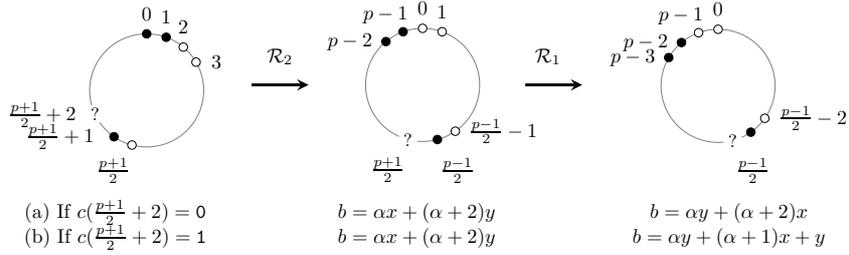
\begin{figure}[h!tbp]
\begin{center}\scalebox{0.75}{
\begin{tabular}{ccc}
\begin{tikzpicture}[scale=1]
	  \path[draw,help lines] (0,0) circle (1 cm);
	  \tikzstyle{every node}=[shape=circle,fill=white,draw=black,minimum size=0.5pt,inner sep=0.5pt]
	  \tikzstyle{white}=[shape=circle,fill=white,draw=black,minimum size=0.5pt,inner sep=1.5pt]
	  \tikzstyle{black}=[shape=circle,fill=black,draw=black,minimum size=0.5pt,inner sep=1.5pt]
	  \node[white,label={[label distance=0.1cm]90:$2$}] at (50:1){};	  
	  \node[black,label={[label distance=0.1cm]90:$0$}] at (90:1) {};
	  \node[black,label={[label distance=0.1cm]90:$1$}] at (70:1) {};
	  \node[white,label={[label distance=0.1cm]0:$3$}] at (30:1) {};
	  \node[black,label={[label distance=0.2cm]180:${\frac{p+1}{2}+1}$}] at (235:1) {};
	  \node[white,label={[label distance=0.1cm]255:$\frac{p+1}{2}$}] at (255:1) {};
	 \node[draw=white,fill=white,label={[label distance=0.1cm]180:${\frac{p+1}{2}+2}$}] at (205:1){\small{?}};	
    \end{tikzpicture}
    &     \begin{tikzpicture}[scale=1]
		  \node at (-2.5,0.5) {$\mathcal{R}_2$};
   	  \draw[arrows={-stealth}, line width=1.5pt] (-3,0) to (-2,0);
   	  \begin{scope}[xshift=4.2cm]
   	  \node at (-2,0.5) {$\mathcal{R}_1$};
   	  \draw[arrows={-stealth}, line width=1.5pt] (-2.4,0) to (-1.4,0);
   	  \end{scope}
   	  \path[draw,help lines] (0,0) circle (1 cm);
	  \tikzstyle{every node}=[shape=circle,fill=white,draw=black,minimum size=0.5pt,inner sep=0.5pt]
	  \tikzstyle{white}=[shape=circle,fill=white,draw=black,minimum size=0.5pt,inner sep=1.5pt]
	  \tikzstyle{black}=[shape=circle,fill=black,draw=black,minimum size=0.5pt,inner sep=1.5pt]
	  \node[black,label={[label distance=0.1cm]180:$p-2$}] at (130:1){};	  
	  \node[white,label={[label distance=0.1cm]90:$0$}] at (90:1) {};
	  \node[white,label={[label distance=0.1cm]90:$1$}] at (70:1) {};
	  \node[black,label={[label distance=-0.1cm]110:$p-1$}] at (110:1) {};
	  \node[black,label={[label distance=0.2cm]285:${\frac{p-1}{2}}$}] at (285:1) {};
	  \node[white,label={[label distance=0.1cm]0:$\frac{p-1}{2}-1$}] at (305:1) {};
	 \node[draw=white,fill=white,label={[label distance=0.1cm]255:${\frac{p+1}{2}}$}] at (255:1){\small{?}};	
    \end{tikzpicture}
	  &
	  \begin{tikzpicture}[scale=1]
   	  \path[draw,help lines] (0,0) circle (1 cm);
	  \tikzstyle{every node}=[shape=circle,fill=white,draw=black,minimum size=0.5pt,inner sep=0.5pt]
	  \tikzstyle{white}=[shape=circle,fill=white,draw=black,minimum size=0.5pt,inner sep=1.5pt]
	  \tikzstyle{black}=[shape=circle,fill=black,draw=black,minimum size=0.5pt,inner sep=1.5pt]
	  \node[black,label={[label distance=0.1cm]180:$p-2$}] at (130:1){};	  
	  \node[white,label={[label distance=0.1cm]90:$0$}] at (90:1) {};
	  \node[black,label={[label distance=0.1cm]180:$p-3$}] at (150:1) {};
	  \node[white,label={[label distance=-0.1cm]110:$p-1$}] at (110:1) {};
	  \node[white,label={[label distance=0.1cm]0:${\frac{p-1}{2}-2}$}] at (325:1) {};
	  \node[black] at (305:1) {};
	 \node[draw=white,fill=white,label={[label distance=0.1cm]285:${\frac{p-1}{2}}$}] at (285:1){\small{?}};	
    \end{tikzpicture}\\
    (a) If $c(\frac{p+1}{2}+2)={\tt0}$& $b=\alpha x+(\alpha+2)y$ & $b=\alpha y+(\alpha+2)x$\\
    (b) If $c(\frac{p+1}{2}+2)={\tt1}$& $b=\alpha x+(\alpha+2)y$ & $b=\alpha y+(\alpha+1)x+y$  
\end{tabular}}
\end{center}
\caption{Rotations of the colouring $c$ of a Type1mod cycle with $c(0)=c(1)={\tt 1}$, \mbox{$c(\frac{p+1}{2}+1)={\tt 1}$} and $c(3)=c(\frac{p+1}{2})={\tt 0}$, and their corresponding weighted sums of black vertices depending on the colour $c(\frac{p+1}{2}+2)$.}\label{fig:type1mod_case2_4}
\end{figure}

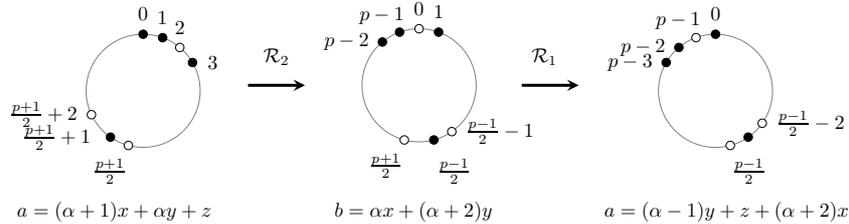
\begin{figure}[h!tbp]
\begin{center}\scalebox{0.75}{
\begin{tabular}{ccc}
\begin{tikzpicture}[scale=1]
	  \path[draw,help lines] (0,0) circle (1 cm);
	  \tikzstyle{every node}=[shape=circle,fill=white,draw=black,minimum size=0.5pt,inner sep=0.5pt]
	  \tikzstyle{white}=[shape=circle,fill=white,draw=black,minimum size=0.5pt,inner sep=1.5pt]
	  \tikzstyle{black}=[shape=circle,fill=black,draw=black,minimum size=0.5pt,inner sep=1.5pt]
	  \node[white,label={[label distance=0.1cm]90:$2$}] at (50:1){};	  
	  \node[black,label={[label distance=0.1cm]90:$0$}] at (90:1) {};
	  \node[black,label={[label distance=0.1cm]90:$1$}] at (70:1) {};
	  \node[black,label={[label distance=0.1cm]0:$3$}] at (30:1) {};
	  \node[black,label={[label distance=0.2cm]180:${\frac{p+1}{2}+1}$}] at (235:1) {};
	  \node[white,label={[label distance=0.1cm]255:$\frac{p+1}{2}$}] at (255:1) {};
	 \node[white,label={[label distance=0.1cm]180:${\frac{p+1}{2}+2}$}] at (205:1){};	
    \end{tikzpicture}
    &     \begin{tikzpicture}[scale=1]
		  \node at (-2.5,0.5) {$\mathcal{R}_2$};
   	  \draw[arrows={-stealth}, line width=1.5pt] (-3,0) to (-2,0);
   	  \begin{scope}[xshift=4.2cm]
   	  \node at (-2,0.5) {$\mathcal{R}_1$};
   	  \draw[arrows={-stealth}, line width=1.5pt] (-2.4,0) to (-1.4,0);
   	  \end{scope}
   	  \path[draw,help lines] (0,0) circle (1 cm);
	  \tikzstyle{every node}=[shape=circle,fill=white,draw=black,minimum size=0.5pt,inner sep=0.5pt]
	  \tikzstyle{white}=[shape=circle,fill=white,draw=black,minimum size=0.5pt,inner sep=1.5pt]
	  \tikzstyle{black}=[shape=circle,fill=black,draw=black,minimum size=0.5pt,inner sep=1.5pt]
	  \node[black,label={[label distance=0.1cm]180:$p-2$}] at (130:1){};	  
	  \node[white,label={[label distance=0.1cm]90:$0$}] at (90:1) {};
	  \node[black,label={[label distance=0.1cm]90:$1$}] at (70:1) {};
	  \node[black,label={[label distance=-0.1cm]110:$p-1$}] at (110:1) {};
	  \node[black,label={[label distance=0.2cm]285:${\frac{p-1}{2}}$}] at (285:1) {};
	  \node[white,label={[label distance=0.1cm]0:$\frac{p-1}{2}-1$}] at (305:1) {};
	 \node[white,label={[label distance=0.1cm]255:${\frac{p+1}{2}}$}] at (255:1){};	
    \end{tikzpicture}
	  &
	  \begin{tikzpicture}[scale=1]
   	  \path[draw,help lines] (0,0) circle (1 cm);
	  \tikzstyle{every node}=[shape=circle,fill=white,draw=black,minimum size=0.5pt,inner sep=0.5pt]
	  \tikzstyle{white}=[shape=circle,fill=white,draw=black,minimum size=0.5pt,inner sep=1.5pt]
	  \tikzstyle{black}=[shape=circle,fill=black,draw=black,minimum size=0.5pt,inner sep=1.5pt]
	  \node[black,label={[label distance=0.1cm]180:$p-2$}] at (130:1){};	  
	  \node[black,label={[label distance=0.1cm]90:$0$}] at (90:1) {};
	  \node[black,label={[label distance=0.1cm]180:$p-3$}] at (150:1) {};
	  \node[white,label={[label distance=-0.1cm]110:$p-1$}] at (110:1) {};
	  \node[white,label={[label distance=0.1cm]0:${\frac{p-1}{2}-2}$}] at (325:1) {};
	  \node[black] at (305:1) {};
	 \node[white,label={[label distance=0.1cm]285:${\frac{p-1}{2}}$}] at (285:1){};	
    \end{tikzpicture}\\
    $a=(\alpha+1)x+\alpha y+z$& $b=\alpha x+(\alpha+2)y$ & $a=(\alpha-1) y+z+(\alpha+2)x$\\
\end{tabular}}
\end{center}
\caption{Rotations of the colouring $c$ of a Type1mod cycle with $c(0)=c(1)={\tt 1}$, \mbox{$c(\frac{p+1}{2}+1)={\tt 1}$} and $c(3)=c(\frac{p+1}{2})=c(\frac{p+1}{2}+2)={\tt 0}$, and their corresponding weighted sums of black vertices which are not all equal.}\label{fig:type1mod_case2_5}
\end{figure}

Therefore, the colouring $c\circ\mathcal{R}_{\frac{p+1}{2}+1}$ has the same configuration as the colouring $c$, i.e., the vertices $0,1$ are black and the vertex $\frac{p+1}{2}$ is white. We can apply the same argument as before. Hence, the colouring $c$ must be $3$-periodic of pattern period ${\tt 110}$ and the number $p$ of vertices is such that $p\equiv0\pmod{3}$.
\end{proof}

\begin{remark}\label{rem:3weights}
In the previous proof, we used the following fact. Let $c$ be a constant $2$-labelling of a cycle of Type1mod with $x\ne y$. If the sum of the weights of black vertices is equal to $a=\alpha_x x+\alpha_y y +z$ with the colouring $c$, where $\alpha_x,\ \alpha_y$ respectively denote the black vertices with weight $x$ and weight $y$, then for any colouring $c\circ\mathcal{R}_j$ such that $c\circ\mathcal{R}_j(0)={\tt 1}$, the weighted sum is $a=\alpha_x x+\alpha_y y +z$ and $\alpha_x$, $\alpha_y$ respectively denote the numbers of black vertices of weight $x$ and weight $y$ with respect to the colouring $c\circ\mathcal{R}_j$. In other words, the number of black vertices of weight $x$ (respectively $y$) is the same for the colouring $c\circ\mathcal{R}_j$ such that $c\circ\mathcal{R}_j(0)={\tt 1}$. 
This fact follows from the uniqueness of the solution $(\lambda,\mu)$ of the system
 $$\left\{
    \begin{array}{l}
    a-z= \lambda x +\mu y\\
    n-1= \lambda + \mu
    \end{array}\right.$$
where $n$ denotes the total number of black vertices (which is known from the colouring $c$).

The same argument holds when the weighted sum is $b=\alpha_x x+\alpha_y y$ and the colouring $c\circ\mathcal{R}_j$ such that $c\circ\mathcal{R}_j(0)={\tt 0}$. 
\end{remark}

Cycles of Type1mod and Type3mod share some similarities. Both types have at most $3$ distinct weights and their non-trivial constant $2$-labellings are the same as shown in the next lemma. We omit the proof here since it follows exactly the same lines as the proof of Lemma~\ref{prop_type5}, but the details can be found in \cite[Appendix~B]{Vandomme--thesis}.

\begin{lemma}\label{prop_type6}
 For cycles $\mathcal{C}_p$ of Type3mod, i.e., $z(xy)^\frac{p-3}{4}xx(yx)^\frac{p-3}{4}$ with $x\neq y$ and 
 $3<p\in\mathbb{N}$, if $ c$ is a non-trivial constant $2$-labelling, then $p\equiv0\pmod{3}$ and $ c$ is $3$-periodic 
 of pattern period ${\tt110}$.
\end{lemma}

Now for cycles of Type2mod and Type4mod with weights $z$, $x$, $y$, $t$, 
if the number $n$ of black vertices and the values
$a:=z+\alpha_x x +\alpha_y y+\alpha_t t$ and $b:=\beta_x x+\beta_y y+\beta_t t$ are known, then the following system 
    \begin{equation*}
    \left\{
    \begin{array}{l}
    a= \lambda x +\mu y +\nu t+z\\
    n= \lambda + \mu+\nu+ 1
    \end{array}\right.
    \qquad
    \left(\text{respectively } 
    \left\{
    \begin{array}{l}
    b= \lambda x +\mu y +\nu t\\
    n= \lambda + \mu+\nu
    \end{array} 
    \right)\right.
    \end{equation*}
does not necessarily have a unique solution $(\lambda,\mu,\nu)=(\alpha_x,\alpha_y,\alpha_t)$ 
(resp. $(\lambda,\mu,\nu)$ $=$  $(\beta_x,\beta_y,\beta_t)$). 
Hence for these cycles, it is be important to make a distinction between the colouring $ c$
and its rotations. 

We first deal with an easy particular case of these cycles that corresponds to a Type2mod cycle with $t=x\neq y$ or to a Type4mod cycle with $t=y\neq x$.

\begin{lemma}\label{prop_type4}
 Let $p\ge4$ be an integer such that $p\equiv0\pmod{2}$. Let $\mathcal{C}_p$ be a cycle of
 of Type2mod with $t=x\ne y$ or of Type4mod with $t=y\neq x$, i.e., $\mathcal{C}_p$ is a cycle represented by $z(xy)^\frac{p-2}{2}x$ with $x\neq y$.
 Any non-trivial constant $2$-labelling $c$ of $\mathcal{C}_p$ is 
 either the alternate colouring, 
 or a colouring such that the number $\alpha_x$ of black vertices of weight $x$ is equal to $\alpha_y+c(0)$ where $\alpha_y$ is the number of black vertices of weight $y$.
\end{lemma}

\begin{proof}
	Let $p\ge4$ be an integer such that $p\equiv0\pmod{2}$ and let $\mathcal{C}_p$ be a cycle 
	represented by $z(xy)^\frac{p-2}{2}x$ with $x\neq y$. Clearly the alternate colouring is a 
    constant $2$-labelling with $a=(\frac{p}{2}-1)y+z$ and $b=\frac{p}{2}x$.

    Now assume that $c$ is a non-trivial constant $2$-labelling of $\mathcal{C}_p$ which is not 
    the alternate colouring. Without loss of generality, we assume that the vertices $0$ and $1$ are both coloured in black. Let $\alpha_x$, $\alpha_y$ denote respectively the number of black vertices with weight $x$ and $y$ for the colouring $c$. We have $a=\alpha_x x +\alpha_y y +z$ as the sum of the weights of black vertices. For the colouring $c\circ\mathcal{R}_1$, the weighted sum is equal to 
    $$a=(\alpha_x-1) y +z + \alpha_y x+x=(\alpha_y+1)x+(\alpha_x-1)y+z.$$ 
    As $x\ne y$, we get $\alpha_x=\alpha_y+1$ and we set $\alpha:=\alpha_y$. 
    
    Let $i$ be the smallest integer in $\{0,\ldots,p-2\}$ such that $c(i+1)={\tt 0}$. The weighted sum for the colouring $c\circ\mathcal{R}_{i}$ is $a=(\alpha+1)x+\alpha y+z$ by hypothesis. Therefore, the weighted sum for the colouring $c\circ\mathcal{R}_{i+1}$ is equal to $b=(\alpha+1)y+\alpha x+x=(\alpha+1)x+(\alpha+1)y$. Moreover, the weighted sum is preserved for the colouring $c\circ \mathcal{R}_{i+2}$ regardless to the colour of the vertex $i+2$:
    $$\left\{\begin{array}{ll}
    b=(\alpha+1)y+(\alpha+1)x &\text{if } c(i+2)={\tt 0}\\
    a=(\alpha)y+z+(\alpha+1)y &\text{if } c(i+2)={\tt 1}.\\
    \end{array}\right.$$

    It follows that the only condition on the constant $2$-labelling $c$ is to be a colouring with $\alpha_x=\alpha_y+1$ if $c(0)={\tt 1}$. Similarly, the condition is $\alpha_x=\alpha_y$ if $c(0)={\tt0}$.
\end{proof}

We now consider the Type2mod cycles in general.

\begin{lemma}\label{prop_type7}
 Let $p\equiv 2\pmod{4}$ with $p>2$ and let $\mathcal{C}_p$ be a weighted cycle of Type2mod represented by $z(xy)^\frac{p-2}{4}t(yx)^\frac{p-2}{4}$ where the weights  $x,y,t$ are not all equal.
  If $c$ is a non-trivial constant $2$-labelling, then $c$ is one of the following colouring  
   \begin{itemize}
  \item alternate,
  \item $\frac{p}{2}$-periodic,
  \item if $x=y$, $\frac{p}{2}$-anti-periodic,
  \item  if $t=x$, any colouring such that the number of black vertices of weight $x$ is equal to the sum of $c(0)$ and the number of black vertices of weight $y$.
\end{itemize}   
\end{lemma}

\begin{proof}
Let $p\equiv 2\pmod{4}$ with $p>2$ and let $\mathcal{C}_p$ be a weighted cycle of Type2mod represented by $z(xy)^\frac{p-2}{4}t(yx)^\frac{p-2}{4}$ where the weights  $x,y,t$ are not all equal. Clearly, the alternate colouring is a constant $2$-labelling with $a=(\frac{p}{2}-1)y+z$ and $b=(\frac{p}{2}-1)x+t$. 

The case where the weights $t$ and $x$ are equal follows from Lemma~\ref{prop_type4}. Hence, we suppose from now on that $t\ne x$. Consider a non-trivial constant $2$-labelling $c$ of $\mathcal{C}_p$ that is not the alternate colouring. Without loss of generality, we may assume that $c(0)=c(1)={\tt1}$. We let $\alpha_x,\alpha_y,\alpha_t$ denote respectively the number of black vertices of weight $x$ and $y$ for the colouring $c$. The sum of the weights of the black vertices is then equal to $a=\alpha_x x+\alpha_y y+\alpha_t t+z$. We consider the colour of the vertex $\frac{p}{2}+1$.

Assume first that $c(\frac{p}{2}+1)={\tt1}$. It follows that $c(\frac{p}{2})={\tt1}$ and $\alpha_t=1$, otherwise the weighted sum is not preserved (Figure~\ref{fig:type2mod_case1_1}). Then for the colouring $c\circ\mathcal{R}_1$, the sum of the weights of the black vertices is
$$a=(\alpha_x-1)y+z+(\alpha_y-1)x+t+x+y=\alpha_y x+\alpha_x y+t+z$$
as under a $1$-rotation, any black vertex with weight $x$ becomes a black vertex of weight $y$, except for the vertex $1$ which becomes the vertex $0$ with weight $z$, and similarly any black vertex with weight $y$ becomes a black vertex of weight $x$, except for the vertex $\frac{p}{2}+1$ which becomes the vertex $\frac{p}{2}$ with weight $t$. If the weights $x$ and $y$ are distinct, then $\alpha_x=\alpha_y$ and we set $\alpha:=\alpha_x$. Otherwise, we denote by $\beta$ the number $\alpha_x+\alpha_y$ of black vertices with weight $x=y$.

\begin{figure}[h!tbp]
\begin{center}
\scalebox{0.74}{
	\begin{tabular}{ccccc}
	 \begin{tikzpicture}[scale=1]
	  \path[draw,help lines] (0,0) circle (1 cm);
	  \tikzstyle{every node}=[shape=circle,fill=white,draw=black,minimum size=0.5pt,inner sep=0.5pt]
	  \tikzstyle{white}=[shape=circle,fill=white,draw=black,minimum size=0.5pt,inner sep=1.5pt]
	  \tikzstyle{black}=[shape=circle,fill=black,draw=black,minimum size=0.5pt,inner sep=1.5pt]	
	  \node[black,label={[label distance=0.1cm]90:$0$}] at (90:1) {};
	  \node[black,label={[label distance=0.1cm]0:$1$}] at (70:1){};	
	  \node[black,label={[label distance=0.2cm]180:${\frac{p}{2}+1}$}] at (250:1) {};
	  \node[white,label={[label distance=0.1cm]270:${\frac{p}{2}}$}] at (270:1) {};
    \end{tikzpicture}
   &\begin{tikzpicture}
	  \node at (1,0){};
	  \node at (1,2) {$\mathcal{R}_{\frac{p}{2}}$};
   	  \draw[arrows={-stealth}, line width=1.5pt] (0.5,1) to (1.5,1);
    \end{tikzpicture}&
  	 \begin{tikzpicture}[scale=1]
	  \path[draw,help lines] (0,0) circle (1 cm);
	  \tikzstyle{every node}=[shape=circle,fill=white,draw=black,minimum size=0.5pt,inner sep=0.5pt]
	  \tikzstyle{white}=[shape=circle,fill=white,draw=black,minimum size=0.5pt,inner sep=1.5pt]
	  \tikzstyle{black}=[shape=circle,fill=black,draw=black,minimum size=0.5pt,inner sep=1.5pt]	
	  \node[white,label={[label distance=0.1cm]90:$0$}] at (90:1) {};
	  \node[black,label={[label distance=0.1cm]0:$1$}] at (70:1){};	
	  \node[black,label={[label distance=0.2cm]180:${\frac{p}{2}+1}$}] at (250:1) {};
	  \node[black,label={[label distance=0.1cm]270:${\frac{p}{2}}$}] at (270:1) {};
    \end{tikzpicture}&
    \begin{tikzpicture}
	  \node at (1,0){};
	  \node at (1,2) {$\mathcal{R}_1$};
   	  \draw[arrows={-stealth}, line width=1.5pt] (0.5,1) to (1.5,1);
    \end{tikzpicture}
    &  	 
    \begin{tikzpicture}[scale=1]
	  \path[draw,help lines] (0,0) circle (1 cm);
	  \tikzstyle{every node}=[shape=circle,fill=white,draw=black,minimum size=0.5pt,inner sep=0.5pt]
	  \tikzstyle{white}=[shape=circle,fill=white,draw=black,minimum size=0.5pt,inner sep=1.5pt]
	  \tikzstyle{black}=[shape=circle,fill=black,draw=black,minimum size=0.5pt,inner sep=1.5pt]	
	  \node[black,label={[label distance=0.1cm]90:$0$}] at (90:1) {};
	  \node[white,label={[label distance=0.1cm]180:$p-1$}] at (110:1){};	
	  \node[black,label={[label distance=0.2cm]0:${\frac{p}{2}-1}$}] at (290:1) {};
	  \node[black,label={[label distance=0.1cm]270:${\frac{p}{2}}$}] at (270:1) {};
    \end{tikzpicture}\\
    $a=\alpha_x x+\alpha_y y+z$ && $b=\alpha_x y+\alpha_y x+t$ && $a=(\alpha_x-1)x+t+(\alpha_y-1)y+z+y$
	\end{tabular}
}
\end{center}
\caption{Rotations of the colouring $c$ of a Type2mod cycle with $c(\frac{p}{2}+1)={\tt1}$, and their corresponding weighted sums of black vertices which are not all equal as $x\ne t$.}\label{fig:type2mod_case1_1}
\end{figure}

Let $i$ be the smallest integer in $\{0,\ldots,\frac{p}{2}-1\}$ such that $ c(i+1)={\tt0}$ and assume that $c(\frac{p}{2}+i)={\tt1}$ for any $\ell\in\{0,\ldots,i\}$ (otherwise, consider the colouring $c\circ \mathcal{R}_{\frac{p}{2}}$ instead of $c$). Then $c(\frac{p}{2}+i+1)={\tt0}$ as depicted in Figure~\ref{fig:type2mod_case1_2}. 

 \begin{figure}[h!tbp]
 \begin{center}
 \scalebox{0.75}{
 \begin{tabular}{c c c}
  \begin{tikzpicture}[scale=1]
	  \path[draw,help lines] (0,0) circle (1 cm);
	  \tikzstyle{every node}=[shape=circle,fill=white,draw=black,minimum size=0.5pt,inner sep=0.5pt]
	  \tikzstyle{white}=[shape=circle,fill=white,draw=black,minimum size=0.5pt,inner sep=1.5pt]
	  \tikzstyle{black}=[shape=circle,fill=black,draw=black,minimum size=0.5pt,inner sep=1.5pt]
	  \tikzstyle{trait}=[dashed, draw=white,line width=0.75pt]
	  \draw[trait] (35:1) arc (35:85:1cm);
	  \draw[trait] (215:1) arc (215:265:1cm);
	  \node[white] at (10:1)[label=right:${i+1}$] {};	  
	  \node[black] at (30:1)[label=right:$i$] {};
	  \node[black] at (90:1) [label=above:$0$]{};
	  \node[black] at (190:1) [label=left:${\frac{p}{2}+i+1}$]{};
	  \node[black] at (210:1) [label=left:${\frac{p}{2}+i}$]{};
	  \node[black] at (270:1) [label=below:$\frac{p}{2}$]{};

    \end{tikzpicture}
   &\begin{tikzpicture}
	  \node at (1,0){};
	  \node at (1,2) {$\mathcal{R}_i$};
   	  \draw[arrows={-stealth}, line width=1.5pt] (0,1) to (2,1);
    \end{tikzpicture}
   &\begin{tikzpicture}[scale=1]
  	  \path[draw,help lines] (0,0) circle (1 cm);
	  \tikzstyle{trait}=[dashed, draw=white,line width=0.75pt]
	  \tikzstyle{every node}=[shape=circle,fill=white,draw=black,minimum size=0.5pt,inner sep=0.5pt]
	  \tikzstyle{black}=[shape=circle,fill=black,draw=black,minimum size=0.5pt,inner sep=1.5pt]
	  \tikzstyle{white}=[shape=circle,fill=white,draw=black,minimum size=0.5pt,inner sep=1.5pt]
	  \draw[trait] (95:1) arc (95:145:1cm);
	  \draw[trait] (275:1) arc (275:325:1cm);
	  \node[white] at (70:1)[label=right:$\phantom{1}1$] {};	  
	  \node[black] at (90:1)[label=above:$0$] {};
	  \node[black] at (150:1) [label=left:${p-i}$]{};
	  \node[black] at (250:1) [label=left:${\frac{p}{2}+1\phantom{1}}$]{};
	  \node[black] at (270:1) [label=below:${\frac{p}{2}}$]{};
	  \node[black] at (330:1) [label=right:${\frac{p}{2}-i}$]{};
    \end{tikzpicture}\\
    (a) $a=\alpha x+\alpha y+z+t$ & & (a) $a=\alpha x+\alpha y+z+t$\\ 
        (b) $a=\beta x+z+t$ & & (b) $a=\beta x+z+t$\\ 
      & &\\
	& & \begin{tikzpicture}
	  \node at (-1.5,0){};	
	  \node at (1,0.5) {$\mathcal{R}_\frac{p}{2}$};
   	  \draw[arrows={-stealth}, line width=1.5pt] (0,1) to (0,0);
    \end{tikzpicture}\\
    \begin{tikzpicture}[scale=1]
  	  \path[draw,help lines] (0,0) circle (1 cm);
	  \tikzstyle{trait}=[dashed, draw=white,line width=0.75pt]
	  \tikzstyle{every node}=[shape=circle,fill=white,draw=black,minimum size=0.5pt,inner sep=0.5pt]
	  \tikzstyle{black}=[shape=circle,fill=black,draw=black,minimum size=0.5pt,inner sep=1.5pt]
	  \tikzstyle{white}=[shape=circle,fill=white,draw=black,minimum size=0.5pt,inner sep=1.5pt]
	  \draw[trait] (115:1) arc (115:145:1cm);
	  \draw[trait] (295:1) arc (295:345:1cm);
	  \node[white] at (90:1)[label=above:$0$] {};	  
	  \node[black] at (110:1)[label=left:${p-1}$] {};
	  \node[black] at (170:1) [label=left:${p-i-1}$]{};
	  \node[black] at (270:1) [label=below:${\frac{p}{2}}$]{};
	  \node[black] at (290:1) [label=right:$\phantom{1}{\frac{p}{2}-1}$]{};
	  \node[black] at (350:1) [label=right:${\frac{p}{2}-i-1}$]{};
    \end{tikzpicture}  &\begin{tikzpicture}
	  \node at (1,0){};
	  \node at (1,2) {$\mathcal{R}_1$};
   	  \draw[arrows={-stealth}, line width=1.5pt] (2,1) to (0,1);
    \end{tikzpicture}
    &  \begin{tikzpicture}[scale=1]
  	  \path[draw,help lines] (0,0) circle (1 cm);
	  \tikzstyle{trait}=[dashed, draw=white,line width=0.75pt]
	  \tikzstyle{every node}=[shape=circle,fill=white,draw=black,minimum size=0.5pt,inner sep=0.5pt]
	  \tikzstyle{black}=[shape=circle,fill=black,draw=black,minimum size=0.5pt,inner sep=1.5pt]
	  \tikzstyle{white}=[shape=circle,fill=white,draw=black,minimum size=0.5pt,inner sep=1.5pt]
	  \draw[trait] (95:1) arc (95:145:1cm);
	  \draw[trait] (275:1) arc (275:325:1cm);
	  \node[black] at (70:1)[label=right:$\phantom{1}1$] {};	  
	  \node[black] at (90:1)[label=above:$0$] {};
	  \node[black] at (150:1) [label=left:${p-i}$]{};
	  \node[white] at (250:1) [label=left:${\frac{p}{2}+1\phantom{1}}$]{};
	  \node[black] at (270:1) [label=below:${\frac{p}{2}}$]{};
	  \node[black] at (330:1) [label=right:${\frac{p}{2}-i}$]{};
    \end{tikzpicture}  \\
      (a) $a=\alpha x+\alpha y+z+x$ & & (a) $a=\alpha y+\alpha x+t+z$\\ 
     (b) $a=(\beta+1)x+z$ && (b) $a=\beta x+t+z$
 \end{tabular}}
 \end{center}
 \caption{Rotations of the colouring $c$ of a Type2mod cycle $\mathcal{C}_p$ with $c(\frac{p}{2}+i+1)={\tt1}$, and their corresponding weighted sums of black vertices which are not all equal, where the line (a) corresponds to the case $x\ne y$ and the line (b) to the case $x=y$.}\label{fig:type2mod_case1_2}
\end{figure}   

With the colouring $c\circ\mathcal{R}_{i+1}$ we obtain (Figure~\ref{fig:type2mod_case1_3}) a sum of the weights of black vertices equal to 
$$b=\left\{\begin{array}{ll}
(\alpha+1) (x+y) &\text{if } x\ne y\\
(\beta+2) x &\text{if } x=y
\end{array}\right.$$
and the number of black vertices of weight $x$ for the colouring $c\circ\mathcal{R}_{i+1}$ is actually $\alpha+1$ (respectively $\beta+2$) when $x\ne y$ (resp. $x= y$). 
Observe that if $c(i+2)={\tt 0}=c(\frac{p}{2}+i+2)$, then the weighted sum $b$ for the colouring $c\circ\mathcal{R}_{i+2}$ is preserved. 

\begin{figure}[h!tbp]
\begin{center}
\scalebox{0.75}{
 \begin{tabular}{c c c}
  \begin{tikzpicture}[scale=1]
	  \path[draw,help lines] (0,0) circle (1 cm);
	  \tikzstyle{every node}=[shape=circle,fill=white,draw=black,minimum size=0.5pt,inner sep=0.5pt]
	  \tikzstyle{white}=[shape=circle,fill=white,draw=black,minimum size=0.5pt,inner sep=1.5pt]
	  \tikzstyle{black}=[shape=circle,fill=black,draw=black,minimum size=0.5pt,inner sep=1.5pt]
	  \tikzstyle{trait}=[dashed, draw=white,line width=0.75pt]
	  \draw[trait] (35:1) arc (35:85:1cm);
	  \draw[trait] (215:1) arc (215:265:1cm);
	  \node[white] at (10:1)[label=right:${i+1}$] {};	  
	  \node[black] at (30:1)[label=right:$i$] {};
	  \node[black] at (90:1) [label=above:$0$]{};
	  \node[white] at (190:1) [label=left:${\frac{p}{2}+i+1}$]{};
	  \node[black] at (210:1) [label=left:${\frac{p}{2}+i}$]{};
	  \node[black] at (270:1) [label=below:$\frac{p}{2}$]{};
    \end{tikzpicture}
   &\begin{tikzpicture}[scale=1]
   		  \node at (-2.5,0.5) {$\mathcal{R}_i$};
   	  \draw[arrows={-stealth}, line width=1.5pt] (-3,0) to (-2,0);
   	  \begin{scope}[xshift=4.2cm]
   	  \node at (-2,0.5) {$\mathcal{R}_1$};
   	  \draw[arrows={-stealth}, line width=1.5pt] (-2.4,0) to (-1.4,0);
   	  \end{scope}
   	  
  	  \path[draw,help lines] (0,0) circle (1 cm);
	  \tikzstyle{trait}=[dashed, draw=white,line width=0.75pt]
	  \tikzstyle{every node}=[shape=circle,fill=white,draw=black,minimum size=0.5pt,inner sep=0.5pt]
	  \tikzstyle{black}=[shape=circle,fill=black,draw=black,minimum size=0.5pt,inner sep=1.5pt]
	  \tikzstyle{white}=[shape=circle,fill=white,draw=black,minimum size=0.5pt,inner sep=1.5pt]
	  \draw[trait] (95:1) arc (95:145:1cm);
	  \draw[trait] (275:1) arc (275:325:1cm);
	  \node[white] at (70:1)[label=right:$\phantom{1}1$] {};	  
	  \node[black] at (90:1)[label=above:$0$] {};
	  \node[black] at (150:1) [label=left:${p-i}$]{};
	  \node[white] at (250:1) [label=left:${\frac{p}{2}+1\phantom{1}}$]{};
	  \node[black] at (270:1) [label=below:${\frac{p}{2}}$]{};
	  \node[black] at (330:1) [label=right:${\frac{p}{2}-i}$]{};
    \end{tikzpicture}
       &\begin{tikzpicture}[scale=1]
  	  \path[draw,help lines] (0,0) circle (1 cm);
	  \tikzstyle{trait}=[dashed, draw=white,line width=0.75pt]
	  \tikzstyle{every node}=[shape=circle,fill=white,draw=black,minimum size=0.5pt,inner sep=0.5pt]
	  \tikzstyle{black}=[shape=circle,fill=black,draw=black,minimum size=0.5pt,inner sep=1.5pt]
	  \tikzstyle{white}=[shape=circle,fill=white,draw=black,minimum size=0.5pt,inner sep=1.5pt]
	  \draw[trait] (115:1) arc (115:165:1cm);
	  \draw[trait] (295:1) arc (295:345:1cm);
	  \node[white] at (90:1)[label=above:$0$] {};
	  \node[black] at (110:1) [label=left:$p-1$]{};
	  \node[black] at (170:1) [label=left:${p-i-1}$]{};
	  \node[white] at (270:1) [label=below:${\frac{p}{2}}$]{};
	  \node[black] at (290:1) [label=right:${\phantom{1}\frac{p}{2}+1}$]{};
	  \node[black] at (350:1) [label=right:${\frac{p}{2}-i-1}$]{};
	 \end{tikzpicture}\\
    (a) If $x\ne y$ &$a=\alpha x+\alpha y +t+z$ & $b=\alpha y +\alpha x +x +y$\\
    (b) If $x= y$ &$a=\beta x +t+z$ & $b=\beta x +x +x$\\
    \end{tabular}
}
\end{center}
\caption{Rotations of the colouring $c$ of a Type2mod cycle and their corresponding weighted sums of black vertices depending on the equality of the weights $x$ and $y$.}\label{fig:type2mod_case1_3}
\end{figure}
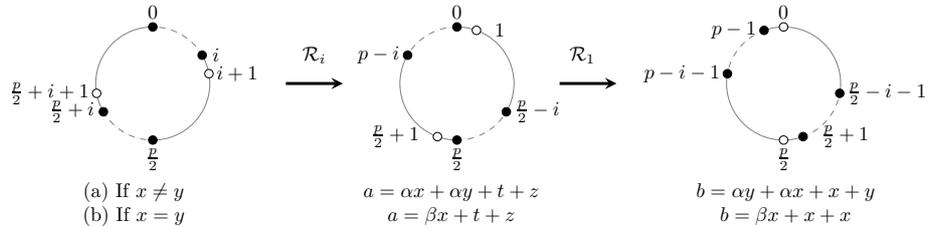

Therefore, let $j$ be the smallest integer in $\{i+1,\ldots,\frac{p}{2}-1\}$ such that $c(j+1)={\tt1}$. Without loss of generality, we assume that $ c(\frac{p}{2}+\ell)={\tt0}$ for all $\ell\in\{i+1,\ldots,j\}$. Then $c\left(\frac{p}{2}+j+1\right)={\tt1}$, otherwise it implies that $x=t$ which is a contradiction (Figure~\ref{fig:type2mod_case1_4}). 
 
 \begin{figure}[h!tbp]
 \begin{center}
 \scalebox{0.75}{
 \begin{tabular}{ccc}
  	  \begin{tikzpicture}[scale=1]
	  \path[draw,help lines] (0,0) circle (1 cm);
	  \tikzstyle{every node}=[shape=circle,fill=white,draw=black,minimum size=0.5pt,inner sep=0.5pt]
	  \tikzstyle{white}=[shape=circle,fill=white,draw=black,minimum size=0.5pt,inner sep=1.5pt]
	  \tikzstyle{black}=[shape=circle,fill=black,draw=black,minimum size=0.5pt,inner sep=1.5pt]
	  \tikzstyle{trait}=[dashed, draw=white,line width=0.75pt]
	  \draw[trait] (35:1) arc (35:85:1cm);
	  \draw[trait] (215:1) arc (215:265:1cm);
	  \draw[trait] (325:1) arc (325:370:1cm);
	  \draw[trait] (145:1) arc (145:185:1cm);
	  \node[white] at (320:1) [label=right:${j}$]{};
	  \node[black] at (305:1) [label={[label distance=0cm]355:${j+1}$}]{};
	  \node[white] at (10:1)[label=right:${i+1}$] {};	  
	  \node[black] at (30:1)[label=right:$i$] {};
	  \node[black] at (90:1) [label=above:$0$]{};
	  \node[white] at (190:1) [label=left:${\frac{p}{2}+i+1}$]{};
	  \node[black] at (210:1) [label=left:${\frac{p}{2}+i}$]{};
	  \node[black] at (270:1) [label=below:$\frac{p}{2}$]{};
	  \node[white] at (140:1) [label={[label distance=0cm]180:}]{};
	  \node[white] at (125:1) [label={[label distance=0cm]180:${\frac{p}{2}+j+1}$}]{};
    \end{tikzpicture}
   &  	   	  \begin{tikzpicture}[scale=1]
      		  \node at (-2.5,0.5) {$\mathcal{R}_j$};
   	  \draw[arrows={-stealth}, line width=1.5pt] (-3,0) to (-2,0);
   	  \begin{scope}[xshift=4.2cm]
   	  \node at (-2,0.5) {$\mathcal{R}_1$};
   	  \draw[arrows={-stealth}, line width=1.5pt] (-2.4,0) to (-1.4,0);
   	  \end{scope}
   	  
	  \path[draw,help lines] (0,0) circle (1 cm);
	  \tikzstyle{every node}=[shape=circle,fill=white,draw=black,minimum size=0.5pt,inner sep=0.5pt]
	  \tikzstyle{white}=[shape=circle,fill=white,draw=black,minimum size=0.5pt,inner sep=1.5pt]
	  \tikzstyle{black}=[shape=circle,fill=black,draw=black,minimum size=0.5pt,inner sep=1.5pt]
	  \tikzstyle{trait}=[dashed, draw=white,line width=0.75pt]
	  \draw[trait] (165:1) arc (165:215:1cm);
	  \draw[trait] (345:1) arc (345:395:1cm);
	  \draw[trait] (455:1) arc (455:500:1cm);
	  \draw[trait] (275:1) arc (275:315:1cm);
 	\node[white] at (90:1) [label=above:${0}$]{};
	  \node[black] at (75:1) [label={[label distance=0cm]5:${1}$}]{};
	  \node[white] at (140:1) {};	  
	  \node[black] at (160:1) {};
	  \node[black] at (220:1) [label=left:$p-j$]{};
	  \node[white] at (320:1){};
	  \node[black] at (340:1) {};
	  \node[black] at (40:1) [label={[label distance=0cm]0:${\frac{p}{2}-j}$}]{};
	  \node[white] at (270:1) [label=below:$\frac{p}{2}$]{};
	  \node[white] at (255:1) [label={[label distance=0cm]200:${\frac{p}{2}+1}$}]{};
    \end{tikzpicture}
       &  	   	  \begin{tikzpicture}[scale=1]
	  \path[draw,help lines] (0,0) circle (1 cm);
	  \tikzstyle{every node}=[shape=circle,fill=white,draw=black,minimum size=0.5pt,inner sep=0.5pt]
	  \tikzstyle{white}=[shape=circle,fill=white,draw=black,minimum size=0.5pt,inner sep=1.5pt]
	  \tikzstyle{black}=[shape=circle,fill=black,draw=black,minimum size=0.5pt,inner sep=1.5pt]
	  \tikzstyle{trait}=[dashed, draw=white,line width=0.75pt]
	  \draw[trait] (180:1) arc (180:230:1cm);
	  \draw[trait] (360:1) arc (360:410:1cm);
	  \draw[trait] (470:1) arc (470:515:1cm);
	  \draw[trait] (290:1) arc (290:330:1cm);
 	\node[white] at (105:1) [label={[label distance=-0.1cm]170:${p-1}$}]{};
	  \node[black] at (90:1) [label=above:${0}$]{};
	  \node[white] at (155:1) {};	  
	  \node[black] at (175:1) {};
	  \node[black] at (235:1) [label=left:$p-j-1$]{};
	  \node[white] at (335:1){};
	  \node[black] at (355:1) {};
	  \node[black] at (55:1) [label={[label distance=0cm]0:${\frac{p}{2}-j-1}$}]{};
	  \node[white] at (285:1) [label={[label distance=0cm]350:${\frac{p}{2}-1}$}]{};
	  \node[white] at (270:1) [label=below:$\frac{p}{2}$]{};
    \end{tikzpicture}\\
 (a) $a=\alpha x+\alpha y +t+z$  & $b=(\alpha+1)(x+y)$ & $a=\alpha y +z+(\alpha+1)x$\\
 (b) $a=\beta x +t+z$ & $b=(\beta+2)x$ & $a=(\beta+1)x+z$
 \end{tabular}
 }
 \end{center}
 \caption{Rotations of the colouring $c$ of a Type2mod cycle with $c(j+1)\neq c(\frac{p}{2}+j+1)={\tt0}$, and their corresponding weighted sums of black vertices which are not all equal, where the line (a) corresponds to the case $x\ne y$ and the line (b) to the case $x=y$.}\label{fig:type2mod_case1_4}
 \end{figure}
	
Consequently, the sum of the weights of the black vertices for the colouring $c\circ\mathcal{R}_{j+1}$ is $a=\alpha x+\alpha y+z+t$ (respectively $a=\beta x+t+z$) if the weights $x$ and $y$ are distinct (resp. equal). Moreover, the colourings $c$ and $c\circ\mathcal{R}_{j+1}$ present the same configuration as $c(j+1)={\tt1}=c(\frac{p}{2}+j+1)$ and as the weighted sums are equal. Hence, we can apply the same reasoning given before for $c$ to the colouring $c\circ\mathcal{R}_{j+1}$.
It follows that the colouring $c$ is $\frac{p}{2}$-periodic. In particular, we have the following weighted sums
  $$\left\{\begin{array}{ll}
  a=\alpha (x+ y)+z+t\text{ and } b=(\alpha+1)(x+y) &\text{if }x\ne y\\
  a=\beta x+t+z\text{ and } b=(\beta+2)x&\text{if }x=y
  \end{array}\right.$$
with $\beta$ even.

\medskip

Assume now that $c(\frac{p}{2}+1)={\tt0}$. It follows that $c(\frac{p}{2})={\tt0}$. Indeed, suppose that $c(\frac{p}{2})={\tt1}$, i.e., $\alpha_t=1$. If the weights $x$ and $y$ are distinct, then the weighted sum $a=\alpha_x x+\alpha_y y+t+z$ for the colouring $c$ implies that the weighted sum for the colouring $c\circ\mathcal{R}_{\frac{p}{2}}$ is equal to $a=\alpha_x y+\alpha_y x+t+z.$
Hence, $\alpha_x=\alpha_y$ as $c$ is a constant $2$-labelling. Then the weighted sum for the colouring $c\circ\mathcal{R}_1$ is given by
$$a=(\alpha_x-1)y+z+\alpha_x x+y+x=(\alpha_x+1)x+\alpha_x y +z$$
which is not equal to the initial weighted sum as $t\ne x$. This is a contradiction.
Now, if the weights $x$ and $y$ are equal, then the weighted sum $a=(\alpha_x+\alpha_y)x+t+z$ for the colouring $c$ implies that the weighted sum for the colouring $c\circ\mathcal{R}_{1}$ is equal to 
$$a=(\alpha_x+\alpha_y-1)x+z+x+x=(\alpha_x+\alpha_y+1)x+z$$
which is a contradiction (as $t\ne x$).

So $c(\frac{p}{2}+1)={\tt0}=c(\frac{p}{2})$ and $\alpha_t=0$. For the colouring $c\circ\mathcal{R}_1$, we obtain the weighted sum $a=(\alpha_x-1)y+z+\alpha_y x+ x$ as depicted in Figure~\ref{fig:type2mod_case2_1}. Hence, $\alpha_x$ must be equal to $\alpha_y+1$ if the weights $x$ and $y$ are distinct. In this case, we set $\alpha=\alpha_y$. In the case where $x=y$, we simply set $\beta=\alpha_x+\alpha_y$. Hence,
$$\left\{\begin{array}{ll}
a=(\alpha+1)x+\alpha y+z &\text{if }x\ne y\\
a=\beta x+z&\text{if }x=y.
\end{array} \right.$$
We obtain the following weighted sum (Figure~\ref{fig:type2mod_case2_1}) for the colouring $c\circ\mathcal{R}_{\frac{p}{2}}$ 
$$\left\{\begin{array}{ll}
b=(\alpha+1)y+\alpha x+t &\text{if }x\ne y\\
b=\beta x+t&\text{if }x=y.
\end{array} \right.$$

\begin{figure}[h!tbp]
\begin{center}
\scalebox{0.75}{
	\begin{tabular}{ccccc}
	 \begin{tikzpicture}[scale=1]
	  \path[draw,help lines] (0,0) circle (1 cm);
	  \tikzstyle{every node}=[shape=circle,fill=white,draw=black,minimum size=0.5pt,inner sep=0.5pt]
	  \tikzstyle{white}=[shape=circle,fill=white,draw=black,minimum size=0.5pt,inner sep=1.5pt]
	  \tikzstyle{black}=[shape=circle,fill=black,draw=black,minimum size=0.5pt,inner sep=1.5pt]	
	  \node[white,label={[label distance=0.1cm]90:$0$}] at (90:1) {};
	  \node[white,label={[label distance=0.1cm]0:$1$}] at (70:1){};	
	  \node[black,label={[label distance=0.2cm]180:${\frac{p}{2}+1}$}] at (250:1) {};
	  \node[black,label={[label distance=0.1cm]270:${\frac{p}{2}}$}] at (270:1) {};
    \end{tikzpicture}
   &\begin{tikzpicture}
	  \node at (1,0){};
	  \node at (1,2) {$\mathcal{R}_{\frac{p}{2}}$};
   	  \draw[arrows={-stealth}, line width=1.5pt] (1.5,1) to (0.5,1);
    \end{tikzpicture}&
  	 \begin{tikzpicture}[scale=1]
	  \path[draw,help lines] (0,0) circle (1 cm);
	  \tikzstyle{every node}=[shape=circle,fill=white,draw=black,minimum size=0.5pt,inner sep=0.5pt]
	  \tikzstyle{white}=[shape=circle,fill=white,draw=black,minimum size=0.5pt,inner sep=1.5pt]
	  \tikzstyle{black}=[shape=circle,fill=black,draw=black,minimum size=0.5pt,inner sep=1.5pt]	
	  \node[black,label={[label distance=0.1cm]90:$0$}] at (90:1) {};
	  \node[black,label={[label distance=0.1cm]0:$1$}] at (70:1){};	
	  \node[white,label={[label distance=0.2cm]180:${\frac{p}{2}+1}$}] at (250:1) {};
	  \node[white,label={[label distance=0.1cm]270:${\frac{p}{2}}$}] at (270:1) {};
    \end{tikzpicture}&
    \begin{tikzpicture}
	  \node at (1,0){};
	  \node at (1,2) {$\mathcal{R}_1$};
   	  \draw[arrows={-stealth}, line width=1.5pt] (0.5,1) to (1.5,1);
    \end{tikzpicture}
    &  	 
    \begin{tikzpicture}[scale=1]
	  \path[draw,help lines] (0,0) circle (1 cm);
	  \tikzstyle{every node}=[shape=circle,fill=white,draw=black,minimum size=0.5pt,inner sep=0.5pt]
	  \tikzstyle{white}=[shape=circle,fill=white,draw=black,minimum size=0.5pt,inner sep=1.5pt]
	  \tikzstyle{black}=[shape=circle,fill=black,draw=black,minimum size=0.5pt,inner sep=1.5pt]	
	  \node[black,label={[label distance=0.1cm]90:$0$}] at (90:1) {};
	  \node[black,label={[label distance=0.1cm]180:$p-1$}] at (110:1){};	
	  \node[white,label={[label distance=0.2cm]0:${\frac{p}{2}-1}$}] at (290:1) {};
	  \node[white,label={[label distance=0.1cm]270:${\frac{p}{2}}$}] at (270:1) {};
    \end{tikzpicture}\\
    $b=\alpha_x y+\alpha_y x+t$ && $a=\alpha_x x+\alpha_y y+z$ && $a=(\alpha_x-1)y+z+\alpha_y x+x$
	\end{tabular}
}
\end{center}
\caption{Rotations of the colouring $c$ of a Type2mod cycle with $c(\frac{p}{2})={\tt0}=c(\frac{p}{2}+1)$, and their corresponding weighted sums of black vertices.}\label{fig:type2mod_case2_1}
\end{figure}

      Let $i$ be the smallest integer in $\{0,\ldots,\frac{p}{2}-1\}$ such that $c(i+1)={\tt0}$. 
      We may assume that $c\left(\frac{p}{2}+\ell\right)={\tt0}$ for all $\ell\in\{0,\ldots,i\}.$      
      Then $c(\frac{p}{2}+i+1)={\tt1}$, otherwise we obtain a contradiction as $x\ne t$ (Figure~\ref{fig:type2mod_case2_2}).

\begin{figure}[h!tbp]
\begin{center}
\scalebox{0.75}{
 \begin{tabular}{c c c}
  \begin{tikzpicture}[scale=1]
	  \path[draw,help lines] (0,0) circle (1 cm);
	  \tikzstyle{every node}=[shape=circle,fill=white,draw=black,minimum size=0.5pt,inner sep=0.5pt]
	  \tikzstyle{white}=[shape=circle,fill=white,draw=black,minimum size=0.5pt,inner sep=1.5pt]
	  \tikzstyle{black}=[shape=circle,fill=black,draw=black,minimum size=0.5pt,inner sep=1.5pt]
	  \tikzstyle{trait}=[dashed, draw=white,line width=0.75pt]
	  \draw[trait] (35:1) arc (35:85:1cm);
	  \draw[trait] (215:1) arc (215:265:1cm);
	  \node[white] at (10:1)[label=right:${i+1}$] {};	  
	  \node[black] at (30:1)[label=right:$i$] {};
	  \node[black] at (90:1) [label=above:$0$]{};
	  \node[white] at (190:1) [label=left:${\frac{p}{2}+i+1}$]{};
	  \node[white] at (210:1) [label=left:${\frac{p}{2}+i}$]{};
	  \node[white] at (270:1) [label=below:$\frac{p}{2}$]{};
    \end{tikzpicture}
   &\begin{tikzpicture}[scale=1]
   		  \node at (-2.5,0.5) {$\mathcal{R}_i$};
   	  \draw[arrows={-stealth}, line width=1.5pt] (-3,0) to (-2,0);
   	  \begin{scope}[xshift=4.2cm]
   	  \node at (-2,0.5) {$\mathcal{R}_1$};
   	  \draw[arrows={-stealth}, line width=1.5pt] (-2.4,0) to (-1.4,0);
   	  \end{scope}
   	  
  	  \path[draw,help lines] (0,0) circle (1 cm);
	  \tikzstyle{trait}=[dashed, draw=white,line width=0.75pt]
	  \tikzstyle{every node}=[shape=circle,fill=white,draw=black,minimum size=0.5pt,inner sep=0.5pt]
	  \tikzstyle{black}=[shape=circle,fill=black,draw=black,minimum size=0.5pt,inner sep=1.5pt]
	  \tikzstyle{white}=[shape=circle,fill=white,draw=black,minimum size=0.5pt,inner sep=1.5pt]
	  \draw[trait] (95:1) arc (95:145:1cm);
	  \draw[trait] (275:1) arc (275:325:1cm);
	  \node[white] at (70:1)[label=right:$\phantom{1}1$] {};	  
	  \node[black] at (90:1)[label=above:$0$] {};
	  \node[black] at (150:1) [label=left:${p-i}$]{};
	  \node[white] at (250:1) [label=left:${\frac{p}{2}+1\phantom{1}}$]{};
	  \node[white] at (270:1) [label=below:${\frac{p}{2}}$]{};
	  \node[white] at (330:1) [label=right:${\frac{p}{2}-i}$]{};
    \end{tikzpicture}
       &\begin{tikzpicture}[scale=1]
  	  \path[draw,help lines] (0,0) circle (1 cm);
	  \tikzstyle{trait}=[dashed, draw=white,line width=0.75pt]
	  \tikzstyle{every node}=[shape=circle,fill=white,draw=black,minimum size=0.5pt,inner sep=0.5pt]
	  \tikzstyle{black}=[shape=circle,fill=black,draw=black,minimum size=0.5pt,inner sep=1.5pt]
	  \tikzstyle{white}=[shape=circle,fill=white,draw=black,minimum size=0.5pt,inner sep=1.5pt]
	  \draw[trait] (115:1) arc (115:165:1cm);
	  \draw[trait] (295:1) arc (295:345:1cm);
	  \node[white] at (90:1)[label=above:$0$] {};
	  \node[black] at (110:1) [label=left:$p-1$]{};
	  \node[black] at (170:1) [label=left:${p-i-1}$]{};
	  \node[white] at (270:1) [label=below:${\frac{p}{2}}$]{};
	  \node[white] at (290:1) [label=right:${\phantom{1}\frac{p}{2}+1}$]{};
	  \node[white] at (350:1) [label=right:${\frac{p}{2}-i-1}$]{};
	 \end{tikzpicture}\\
	(a) If $x\ne y$&	 $a=(\alpha+1)x+\alpha y+z$ & $b=(\alpha+1) y+\alpha x+ x$\\
	(b) If $x=y$&	 $a=\beta x+z$ & $b=\beta x+x$
	 \end{tabular}
	 }
	 \end{center}
	 \caption{Rotations of the colouring $c$ of a Type2mod cycle and their corresponding weighted sums of black vertices depending on the equality of the weights $x$ and $y$.}\label{fig:type2mod_case2_2}
	 \end{figure}
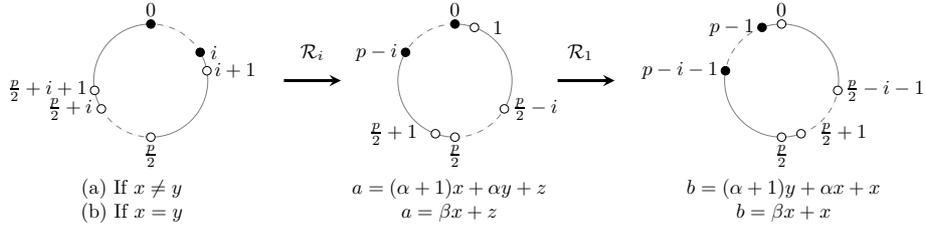
 
 Observe that if $c(i+2)={\tt0}$ and $c(\frac{p}{2}+i+2)={\tt1}$, then the weighted sum $b$ for the colouring $c\circ\mathcal{R}_{i+2}$ is preserved. Therefore, let $j$ be the smallest integer in $\{i+1,\ldots,\frac{p}{2}+i\}$ such that $ c(j+1)={\tt1}$. Without loss of generality, we suppose that $ c\left(\frac{p}{2}+\ell\right)={\tt1}$ for all $\ell\in\{i+1,\ldots,j\}$. It follows that $c\left(\frac{p}{2}+j+1\right)={\tt 0}$. Indeed, $ c\left(\frac{p}{2}+j+1\right)={\tt1}$ leads to a contradiction as $x\ne t$ (Figure~\ref{fig:type2mod_case2_3}).
      
  \begin{figure}[h!tbp]
 \begin{center}
 \scalebox{0.75}{
 \begin{tabular}{ccc}
  	  \begin{tikzpicture}[scale=1]
	  \path[draw,help lines] (0,0) circle (1 cm);
	  \tikzstyle{every node}=[shape=circle,fill=white,draw=black,minimum size=0.5pt,inner sep=0.5pt]
	  \tikzstyle{white}=[shape=circle,fill=white,draw=black,minimum size=0.5pt,inner sep=1.5pt]
	  \tikzstyle{black}=[shape=circle,fill=black,draw=black,minimum size=0.5pt,inner sep=1.5pt]
	  \tikzstyle{trait}=[dashed, draw=white,line width=0.75pt]
	  \draw[trait] (35:1) arc (35:85:1cm);
	  \draw[trait] (215:1) arc (215:265:1cm);
	  \draw[trait] (325:1) arc (325:370:1cm);
	  \draw[trait] (145:1) arc (145:185:1cm);
	  \node[white] at (320:1) [label=right:${j}$]{};
	  \node[black] at (305:1) [label={[label distance=0cm]355:${j+1}$}]{};
	  \node[white] at (10:1)[label=right:${i+1}$] {};	  
	  \node[black] at (30:1)[label=right:$i$] {};
	  \node[black] at (90:1) [label=above:$0$]{};
	  \node[black] at (190:1) [label=left:${\frac{p}{2}+i+1}$]{};
	  \node[white] at (210:1) [label=left:${\frac{p}{2}+i}$]{};
	  \node[white] at (270:1) [label=below:$\frac{p}{2}$]{};
	  \node[black] at (140:1) [label={[label distance=0cm]180:}]{};
	  \node[black] at (125:1) [label={[label distance=0cm]180:${\frac{p}{2}+j+1}$}]{};
    \end{tikzpicture}
   &  	   	  \begin{tikzpicture}[scale=1]
      		  \node at (-2.5,0.5) {$\mathcal{R}_j$};
   	  \draw[arrows={-stealth}, line width=1.5pt] (-3,0) to (-2,0);
   	  \begin{scope}[xshift=4.2cm]
   	  \node at (-2,0.5) {$\mathcal{R}_1$};
   	  \draw[arrows={-stealth}, line width=1.5pt] (-2.4,0) to (-1.4,0);
   	  \end{scope}
   	  
	  \path[draw,help lines] (0,0) circle (1 cm);
	  \tikzstyle{every node}=[shape=circle,fill=white,draw=black,minimum size=0.5pt,inner sep=0.5pt]
	  \tikzstyle{white}=[shape=circle,fill=white,draw=black,minimum size=0.5pt,inner sep=1.5pt]
	  \tikzstyle{black}=[shape=circle,fill=black,draw=black,minimum size=0.5pt,inner sep=1.5pt]
	  \tikzstyle{trait}=[dashed, draw=white,line width=0.75pt]
	  \draw[trait] (165:1) arc (165:215:1cm);
	  \draw[trait] (345:1) arc (345:395:1cm);
	  \draw[trait] (455:1) arc (455:500:1cm);
	  \draw[trait] (275:1) arc (275:315:1cm);
 	\node[white] at (90:1) [label=above:${0}$]{};
	  \node[black] at (75:1) [label={[label distance=0cm]5:${1}$}]{};
	  \node[white] at (140:1) {};	  
	  \node[black] at (160:1) {};
	  \node[black] at (220:1) [label=left:$p-j$]{};
	  \node[black] at (320:1){};
	  \node[white] at (340:1) {};
	  \node[white] at (40:1) [label={[label distance=0cm]0:${\frac{p}{2}-j}$}]{};
	  \node[black] at (270:1) [label=below:$\frac{p}{2}$]{};
	  \node[black] at (255:1) [label={[label distance=0cm]200:${\frac{p}{2}+1}$}]{};
    \end{tikzpicture}
       &  	   	  \begin{tikzpicture}[scale=1]
	  \path[draw,help lines] (0,0) circle (1 cm);
	  \tikzstyle{every node}=[shape=circle,fill=white,draw=black,minimum size=0.5pt,inner sep=0.5pt]
	  \tikzstyle{white}=[shape=circle,fill=white,draw=black,minimum size=0.5pt,inner sep=1.5pt]
	  \tikzstyle{black}=[shape=circle,fill=black,draw=black,minimum size=0.5pt,inner sep=1.5pt]
	  \tikzstyle{trait}=[dashed, draw=white,line width=0.75pt]
	  \draw[trait] (180:1) arc (180:230:1cm);
	  \draw[trait] (360:1) arc (360:410:1cm);
	  \draw[trait] (470:1) arc (470:515:1cm);
	  \draw[trait] (290:1) arc (290:330:1cm);
 	\node[white] at (105:1) [label={[label distance=-0.1cm]170:${p-1}$}]{};
	  \node[black] at (90:1) [label=above:${0}$]{};
	  \node[white] at (155:1) {};	  
	  \node[black] at (175:1) {};
	  \node[black] at (235:1) [label=left:$p-j-1$]{};
	  \node[black] at (335:1){};
	  \node[white] at (355:1) {};
	  \node[white] at (55:1) [label={[label distance=0cm]0:${\frac{p}{2}-j-1}$}]{};
	  \node[black] at (285:1) [label={[label distance=0cm]350:${\frac{p}{2}-1}$}]{};
	  \node[black] at (270:1) [label=below:$\frac{p}{2}$]{};
    \end{tikzpicture}\\
 (a) $a=(\alpha+1) x+\alpha y +z$  & $b=\alpha x+(\alpha+1)y+t$ & $a=(\alpha-1)y+z+\alpha x+t+y$\\
 (b) $a=\beta x +z$ & $b=\beta x+t$ & $a=(\beta-2) x+z+t+x$
 \end{tabular}
 }
 \end{center}
 \caption{Rotations of the colouring $c$ of a Type2mod cycle with $c(j+1)\neq c(\frac{p}{2}+j+1)={\tt0}$, and their corresponding weighted sums of black vertices which are not all equal, where the line (a) corresponds to the case $x\ne y$ and the line (b) to the case $x=y$.}\label{fig:type2mod_case2_3}
 \end{figure}     
 
 Therefore, the sum of the weights of the black vertices for the colouring $ c\circ\mathcal{R}_{j+1}$ is $a=(\alpha+1)x+\alpha y +z$ (respectively $a=\beta x+z$) if the weights $x$ and $y$ are distinct (resp. equal). Hence, the colourings $c$ and $c\circ\mathcal{R}_{j+1}$ present the same configuration as $c(j+1)={\tt1}$ and $c(\frac{p}{2}+j+1)={\tt0}$ and as the weighted sums are equal. It follows that the colouring $c$ is $\frac{p}{2}$-anti-periodic. 

If the weights $x$ and $y$ are distinct, then the number of black vertices is equal to $2\alpha+2=\frac{p}{2}$. It means that $\frac{p}{2}$ is even which is a contradiction as $p\equiv2\pmod{4}$. Thus, there doest not exist a constant $2$-labelling in this case.

If the weights $x$ and $y$ are equal, then any $\frac{p}{2}$-anti-periodic colouring is a constant $2$ labelling with 
$$a=\left(\frac{p}{2}-1\right) x+z\text{ and } b=\left(\frac{p}{2}-1\right) x +t.$$

\end{proof}

The last type of cycles is similar to Type2mod cycles. Hence, the proof of the following lemma is similar to the proof of Lemma~\ref{prop_type7}. Details of the proof are available in \cite[Appendix~B]{Vandomme--thesis}.

\begin{lemma}\label{prop_type8}
  Let $p\equiv 4\pmod{4}$ with $p>4$ and let $\mathcal{C}_p$ be a weighted cycle of Type4mod represented by $z(xy)^\frac{p-4}{4}xtx(yx)^\frac{p-4}{4}$ where the weights  $x,y,t$ are not all equal.
  If $c$ is a non-trivial constant $2$-labelling, then $ c$ is one of the following colouring
   \begin{itemize}
  \item alternate,
  \item $\frac{p}{2}$-anti-periodic,
  \item $\frac{p}{2}$-periodic if $x=y$; $\frac{p}{2}$-periodic and such that the numbers of black vertices of weight $x$ and $y$ are equal when $ c(0)={\tt0}$ if $y\ne x$,
  \item if $t=\frac{p}{4}x+(1-\frac{p}{4})y$, $c$ can be moreover such that $c(i)= c(i+\frac{p}{2})={\tt1}$ for all even $i\in\{0,\ldots,\frac{p}{2}-1\}$ and $ c(i)\ne c(i+\frac{p}{2})$ for all odd $i\in\{0,\ldots,\frac{p}{2}-1\}$ (up to a $1$-rotation).
  \end{itemize}

\end{lemma}
Using all the previous lemmas, we can now prove our main theorem.

\begin{theorem}\label{thm:types}
  Let $ c$ be a non-trivial constant $2$-labelling of a cycle $\mathcal{C}_p$ of Type 0, Type1mod, Type2mod, Type3mod or Type4mod with 
  $A=\{\mathcal{R}_k\mid k\in \mathbb{Z}\}$ and $v=0$.
  Let $a=\sum_{\{u\in V\mid c\circ\xi(u)={\tt1}\}} w(u)$ and 
  $b= \sum_{\{u\in V\mid c\circ\xi'(u)={\tt1}\}} w(u)$ for $\xi\in A_{\tt 1}, \xi'\in A_{\tt 0}$. 
 Then the possible values of the constants $a$ and $b$ are given in the following table.
  \begin{center}
  {\small
  \cellspacetoplimit3pt
  \cellspacebottomlimit3pt

\begin{tabular}{Sc| Sc Sc Sl}
  Type& Value of $a$& Value of $b$& Condition on parameters\\ \hline
  0&$\alpha x+z$& $(\alpha+1)x$& $\alpha\in\{0,\ldots,p-2\}$\\ \hline
  1mod&$\frac{p}{3}x+(\frac{p}{3}-1)y+z$ &$(\frac{p}{3}-1)x+(\frac{p}{3}+1)y$ & $p\equiv 0\pmod{3}$\\ \hline
  3mod& $\frac{p}{3}x+(\frac{p}{3}-1)y+z$ &$(\frac{p}{3}+1)x+(\frac{p}{3}-1)y$ & $p\equiv 0\pmod{3}$\\ \hline
  2mod& $(\frac{p}{2}-1)y+z$& $(\frac{p}{2}-1)x+t$&\\
  & $\alpha(x+y)+t+z$ &$(\alpha+1)(x+y)$  &$\alpha\in\{0,\ldots,\frac{p}{2}-1\}$\\ \hline
  4mod &$(\frac{p}{2}-2)y+z+t$  & $\frac{p}{2}x$&\\
  & $(2\alpha+2)x+2\alpha y+z+t$& $(2\alpha+2)(x+y)$&$\alpha\in\{0,\ldots,\frac{p}{4}-1\}$\\
  &$\frac{p}{4}x+(\frac{p}{4}-1)y+z$&$\frac{p}{4}x+(\frac{p}{4}-1)y+t$&\\
  &$\frac{p}{2}x+(\frac{p}{4}-1)y+z$ &$\frac{3p}{4}x$&$t=\frac{p}{4}x+(1-\frac{p}{4})y$\\
  &$(\frac{p}{4}-1)y+z$ &$\frac{p}{4}x$&$t=\frac{p}{4}x+(1-\frac{p}{4})y$\\
  &$2\alpha x+t+z$&$2(\alpha+1)x$&$\alpha\in\{0,\ldots,\frac{p}{2}-2\}$, $x=y$\\
  &$(\alpha+1)x+\alpha y+z$ &$(\alpha+1)(x+y)$ &$\alpha\in\{0,\ldots,\frac{p}{2}-2\}$, $t=y$\\
  \end{tabular}
  }
  \end{center}
 \end{theorem}

\section{Projection and folding method}\label{sec:projection_folding}

In this section, we present a method that allows us to translate specific colouring problems of the infinite grid in terms of constant $2$-labellings of weighted cycles. We first give the instance  of such problems. Let $t$ and $p$ be integers and let $\mathbf{t}=(t,1)$, $\mathbf{p}=(p,0)$. A \emph{frame} is a set of vertices of a given shape 
where one of the vertices plays a special role and therefore is called the \emph{center of the frame}. Let $a$ and $b$ be non-negative integers.
We consider the problem of deciding whether there exists a $2$-colouring $c$ of the infinite grid such that the colouring is periodic with $c(\mathbf{y}+\mathbf{t})=c(\mathbf{y})=c(\mathbf{y}+\mathbf{p})$ for any $\mathbf{y}\in\mathbb{Z}^2$, and that each frame contains
\begin{itemize}
\item $a$ black vertices if the center of the frame is black,
\item $b$ black vertices if the center of the frame is white.
\end{itemize}
Clearly, if the frames are the balls of radius $r$, then the problem is the same as determining if there exists an $(r,a,b)$-covering code of the infinite grid that is periodic of periods $\mathbf{t}$ and $\mathbf{p}$.

Now, for $\mathbf{t}=(t,1)$, $\mathbf{p}=(p,0)$, consider a $2$-colouring of $\mathbb{Z}^2$ that is periodic of periods $\mathbf{t}$ and $\mathbf{p}$. 
Since $c$ is periodic of period $\mathbf{t}$, the colouring of a line is obtained by doing a translation $\mathbf{t}=(t,1)$ (respectively -$\mathbf{t}=(-t,-1)$) of the colouring of the line below (resp. above). In this case, if we know the colouring of one line and the translation $\mathbf{t}$, then the colouring of the whole grid $\mathbb{Z}^2$ is known.

\subsection*{Projection}

Let $\mathbf{y}\in\mathbb{Z}^2$. Using the translation $\mathbf{t}=(t,1)$, we can project the frame with center $\mathbf{y}$ on the line $L$ containing $\mathbf{y}$. We assume $\mathbf{y}=(0,0)$ to simplify the notation. Let $Trans$ denote the set of all the translated frames of the frame with center $\mathbf{y}$ by a multiple of $\mathbf{t}$. Let $h:L\to\mathbb{N}$ be a map defined by
$$h((i,0))=\#\{ T\in Trans\mid (i,0)\in T\}.$$
The image of the line $L$ by the mapping $h$, denoted by $h(L)$, is called the \emph{projection} of the frame with center $\mathbf{y}$ with translation  $\mathbf{t}=(t,1)$. An example is given in the following section.
Observe that $h((i,0))$ is a finite number for any $i\in\mathbb{N}$, $h$ takes a non-zero value only finitely many times and the number of vertices of a frame is equal to $\sum_{i\in\mathbb{Z}}h((i,0))$. The map $h$ is introduced to count the number of occurrences in the frame with center $\mathbf{y}$ of vertices of $L$, up to translation $\mathbf{t}$.  

\subsection*{Folding}

Using the translation $(p,0)$, we can fold a projection on a cycle of $p$ weighted vertices. Let $L$ be the 
line containing $\mathbf{y}=(0,0)$ and $\{0,\ldots,p-1\}$ be the set of vertices of the cycle $\mathcal{C}_p$. 
We define a map $w: \{0,\ldots,p-1\}\to\mathbb{N}$ such that, for $i\in\{0,\ldots,p-1\}$, 
$$w(i):=\sum_{k\in\mathbb{Z}}h((i+kp,0)).$$
The \emph{folding} of the projection $h(L)$ is the cycle $\mathcal{C}_p$ with vertices $0,\ldots,p-1$ 
of respective weights $w(0),\ldots,w(p-1)$.

\section{Application to $(r,a,b)$-codes of $\mathbb{Z}^2$}\label{sec:application_rabc}

The projection and folding method can be used to find $(r,a,b)$-codes that are periodic. We first give an example, then we characterize the values of $a$ and $b$ of any $(r,a,b)$-code with $r\ge2$ and $|a-b|>4$.

\begin{example}
Consider frames that are balls of radius $r=3$ and set $t=2$, $p=4$. Using the projection and folding method, there exists an $(r,a,b)$-code of the infinite grid that is periodic of periods $\mathbf{t}=(2,1)$ and $\mathbf{p}=(4,0)$ if and only there exists a constant $2$-labelling of the cycle $\mathcal{C}_4$ with weights $w(0)=7,w(1)=w(2)=w(3)=6$. For instance, the colouring $c$ defined by $c(0)=c(1)={\tt 1}$ and $c(2)=c(3)={\tt 0}$ is a constant $2$-labelling. Hence, there exists an $(3,13,12)$-code of $\mathbb{Z}^2$, which is given at the bottom of Figure~\ref{fig:example_projection_folding}.

\begin{figure}[h!tbp]
\begin{center}
\includegraphics[scale=1]{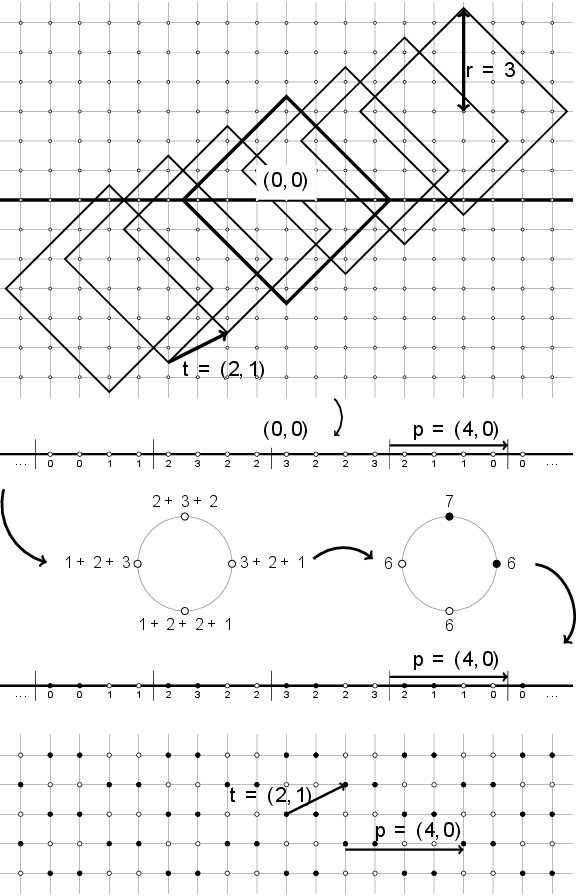}
\end{center}
\caption{Projection and folding of a ball of radius $3$ with respect to the translations $\mathbf{t}=(2,1)$ and $\mathbf{p}=(4,0)$.}\label{fig:example_projection_folding}
\end{figure}
\end{example}

Let $r\ge2$ and $a,b\in\mathbb{N}$ such that $|a-b|>4$. Let $c$ be an $(r,a,b)$-code of $\mathbb{Z}^2$.
By Theorem~\ref{thm_Axe}, $ c$ is a diagonal colouring. Hence, $ c$ is determined 
by the colouring of any horizontal line, e.g. $\{(x_1,0)\mid x_1\in \mathbb{Z}\}$, and by the orientation 
of the monochromatic diagonals in the even and odd sublattices. 

Assume first that the monochromatic diagonals are all parallel. Without loss of generality, 
we can suppose that they are of the type $\{(x_1,x_1+c)\mid x_1\in\mathbb{Z}\}$ with $c\in\mathbb{Z}$.
Indeed, the case where the monochromatic diagonals are of type $\{(x_1,-x_1+c)\mid x_1\in\mathbb{Z}\}$ is similar
since the grid is symmetric.
In this case, if the colouring of a line of $\mathbb{Z}^2$ is known, then the colouring of the line above (resp. below) is 
obtained by doing a translation $\mathbf{t}=(1,1)$ (resp. $-\mathbf{t}$) as 
$ c(\mathbf{x})= c(\mathbf{x}+\mathbf{t})$ for all $\mathbf{x}\in\mathbb{Z}^2$.
So we can apply the projection method.
Moreover, by Theorem~\ref{thm_Axe}, $ c$ is such that $ c(\mathbf{x}+(m,0))= c(\mathbf{x})$
for some $m\in\mathbb{N}$ and all $\mathbf{x}\in\mathbb{Z}^2$. Hence, it is possible to apply the folding method.

Now assume that the monochromatic diagonals are not parallel. We may suppose that the even (resp. odd) 
sublattice is the union of monochromatic diagonals of type $\{(x_1,x_1+c)\mid x_1\in\mathbb{Z}\}$ 
(resp. $\{(x_1,-x_1+c)\mid x_1\in\mathbb{Z}\}$) with $c\in\mathbb{Z}$. We consider an $r$-ball $B_r(\mathbf{y})$ with center $\mathbf{y}$. Observe that a diagonal intersecting the ball contains either $r$ or $r+1$ elements of the ball. 
Moreover two intersecting diagonals belong to 
the same sublattice. Hence, in terms of counting vertices of a particular colour appearing in the ball, it is equivalent 
to consider monochromatic diagonals that are parallel or not. So, we can apply the folding method in both cases.

Therefore, for $r\ge2$ and $|a-b|>4$, there exists an $(r,a,b)$-code of the infinite grid $\mathbb{Z}^2$ if and only if 
there exists a constant $2$-labelling of some cycle $\mathcal{C}_p$, with $v=0$, $A=\{\mathcal{R}_k\mid k\in\mathbb{Z}\}$
and the mapping $w$ defined as before, such that
$$a=\sum_{\{u\in V\mid  c\circ\xi(u)={\tt1}\}}w(u) \text{ and } b= \sum_{\{u\in V\mid c\circ\xi'(u)={\tt1}\}} w(u) 
\quad \forall \xi\in A_{\tt 1},\; \xi'\in A_{\tt 0}.$$

\subsection{Characterization of $(r,a,b)$-codes of $\mathbb{Z}^2$ with $|a-b|>4$ and $r\ge2$}
\renewcommand{\arraystretch}{1.5}
\begin{theorem}\label{thm:codes}
 Let $r, a, b\in\mathbb{N}$ such that $|a-b|>4$ and $r\ge2$. If there exists an $(r,a,b)$-code of $\mathbb{Z}^2$, then the values of $a$ and $b$ are given in the following table   \\
  \begin{tabular}{ccc}
 $a$&$b$& Condition on parameters\\ \hline
 $r+1+\alpha(2r+1)$					& $(\alpha+1)(2r+1)$	&\small{$\alpha\in\{0,\ldots,r-1\}$, $r\equiv0\pmod{2}$}\\
 $(r+1)^2-\alpha(\frac{3r}{2}+1)$	& $r^2+\alpha(\frac{3r}{2}+1)$	& \small{$\alpha\in\{0,1\}$, $r\equiv0\pmod{2}$}\\
 $r+1+(\alpha+1)(2r+1)$				& $(\alpha+1)(2r+1)$	&\small{$\alpha\in\{0,\ldots,r-2\}$, $r\equiv1\pmod{2}$}\\
 $r^2+\alpha\frac{3r+1}{2}$ 		& $(r+1)^2-\alpha\frac{3r+1}{2}$	& \small{$\alpha\in\{0,1\}$, $r\equiv1\pmod{2}$}\\
 $(\alpha+1)\frac{2r^2+2r+2}{3}-1$ 	& $(\alpha+1)\frac{2r^2+2r+2}{3}$ 	& \small{$\alpha\in\{0,1\}$, $r\equiv1\pmod{3}$}\\
 $(\alpha+1)\frac{2r^2+2r}{3}-\frac{r+1}{3}+1$ & $(\alpha+1)\frac{2r^2+2r}{3}+\frac{r+1}{3}$ & \small{$\alpha\in\{0,1\}$, $r\equiv2\pmod{3}$}\\
 $(\alpha+1)\frac{2r^2+2r}{3}+\frac{r}{3}-1$ & $(\alpha+1)\frac{2r^2+2r}{3}-\frac{r}{3}$ &\small{ $\alpha\in\{0,1\}$, $r\equiv0\pmod{3}$}\\
\end{tabular}
\end{theorem}
\renewcommand{\arraystretch}{1}

\begin{proof}
For $r\ge2$ and $|a-b|>4$, Axenovich described all possible $(r,a,b)$-codes (see Theorem~\ref{thm_Axe}) in terms of diagonal colourings. Theorem~\ref{thm_Axe} allows us to apply the projection and folding method in this case.
Let $\mathbf{y}=(0,0)$. We project the ball $B_r(\mathbf{y})$ on the line $L$ using the translation $\mathbf{t}=(1,1)$ and we obtain for an even radius $r$
\begin{equation*}
h((i,0))=\left\{
\begin{array}{l l}
r &\text{if }i\le r\text{ and }i\text{ is odd}\\
r+1 &\text{if }i\le r\text{ and }i\text{ is even}\\
0 & \text{otherwise} 
\end{array}\right.
\end{equation*}
and for an odd radius $r$
\begin{equation*}
h((i,0))=\left\{
\begin{array}{l l}
r+1 &\text{if }i\le r\text{ and }i\text{ is odd}\\
r &\text{if }i\le r\text{ and }i\text{ is even}\\
0 & \text{otherwise.} 
\end{array}\right.
\end{equation*}
Indeed, if $r$ is even, then any diagonal of the even (respectively odd) sublattice 
intersecting the ball contains $r+1$ (resp. $r$) elements of $B_r(\mathbf{y})$.
The other case can be treated similarly. 

\begin{table}
  \begin{center}
  \begin{tabular}{c|c|c}
  &For $r$ even & For $r$ odd\\ \hline
  &Type 0: $p=r+1$ &  Type 0: $p=r$\\
   \begin{tikzpicture}[scale=1]
  	\node at (0,0) {\rotatebox{90}{Colouring 1}};
\end{tikzpicture}&\begin{tikzpicture}[scale=1]
	  \node at (90:1.3) {${r+1}$};
	  \path[draw,help lines] (0,0) circle (1 cm);
	  \tikzstyle{trait}=[dashed]
	  \draw[trait] (-215:1.1) arc (-215:35:1.1cm);
	  \tikzstyle{every node}=[shape=circle,fill=white,draw=black,minimum size=0.5pt,inner sep=1.5pt]
	  \tikzstyle{empty}=[fill=none,draw=none,minimum size=0.5pt,inner sep=1.5pt]
	  \node at (60:1)[label=right:${2r+1}$] {};
	  \node at (90:1) {};
	  \node at (120:1)[label=left:${2r+1}$] {};
	  \node[empty] at (270:1) [label=below:\phantom{t}] {};
  \end{tikzpicture}
  &
    \begin{tikzpicture}[scale=1]
	  \node at (90:1.3) {${3r+2}$};
	  \path[draw,help lines] (0,0) circle (1 cm);
	  \tikzstyle{trait}=[dashed]
	  \draw[trait] (-215:1.1) arc (-215:35:1.1cm);
	  \tikzstyle{every node}=[shape=circle,fill=white,draw=black,minimum size=0.5pt,inner sep=1.5pt]
	  \tikzstyle{empty}=[fill=none,draw=none,minimum size=0.5pt,inner sep=1.5pt]
	  \node at (60:1)[label=right:${2r+1}$] {};
	  \node at (90:1) {};
	  \node at (120:1)[label=left:${2r+1}$] {};
	  \node[empty] at (270:1) [label=below:\phantom{t}] {};
  \end{tikzpicture}\\
\hline
& Type4mod: $p=2r$ &Type4mod: $p=2(r+1)$ \\
  \begin{tikzpicture}[scale=1]
  	\node at (0,0) {\rotatebox{90}{Colouring 2}};
\end{tikzpicture}
&\begin{tikzpicture}[scale=1]
	  \node at (90:1.3) {${r+1}$};
    	  \node at (270:1.3) {${2(r+1)}$};
	  \path[draw,help lines] (0,0) circle (1 cm);
	  \tikzstyle{trait}=[dashed]
	  \draw[trait] (165:1.1) arc (165:215:1.1cm);
	  \draw[trait] (-35:1.1) arc (-35:15:1.1cm);
	  \tikzstyle{every node}=[shape=circle,fill=white,draw=black,minimum size=0.5pt,inner sep=1.5pt]
	  \node at (30:1)[label=right:${r+1}$] {};
	  \node at (60:1)[label=right:$r$] {};
	  \node at (90:1) {};
	  \node at (120:1)[label=left:$r$] {};
	  \node at (150:1)[label=left:${r+1}$] {};
	  \node at (240:1) [label=left:$r$]{};
	  \node at (270:1) {};
	  \node at (300:1) [label=right:$r$]{};
  \end{tikzpicture} &
  \begin{tikzpicture}[scale=1]
  	  \node at (90:1.3) {$r$};
  	  \node at (270:1.3) {$0$};
	  \path[draw,help lines] (0,0) circle (1 cm);
	  \tikzstyle{trait}=[dashed]
	  \draw[trait] (165:1.1) arc (165:215:1.1cm);
	  \draw[trait] (-35:1.1) arc (-35:15:1.1cm);
	  \tikzstyle{every node}=[shape=circle,fill=white,draw=black,minimum size=0.5pt,inner sep=1.5pt]
	  \node at (30:1)[label=right:$r$] {};
	  \node at (60:1)[label=right:${r+1}$] {};
	  \node at (90:1) {};
	  \node at (120:1)[label=left:${r+1}$] {};
	  \node at (150:1)[label=left:$r$] {};
	  \node at (240:1) [label=left:${r+1}$]{};
	  \node at (270:1) {};
	  \node at (300:1) [label=right:${r+1}$]{};
  \end{tikzpicture}\\ \hline
  &Type2mod or Type4mod: $p=r$   & Type2mod or Type4mod: $p=r+1$  \\
  \begin{tikzpicture}[scale=1]
  	\node at (0,0) {\rotatebox{90}{Colouring 3}};
\end{tikzpicture}
&\begin{tikzpicture}[scale=1]
	  \node at (90:1.3) {${3(r+1)}$};
	  \path[draw,help lines] (0,0) circle (1 cm);
	  \tikzstyle{trait}=[dashed]
	  \draw[trait] (-175:1.1) arc (-175:20:1.1cm);
	  \tikzstyle{every node}=[shape=circle,fill=white,draw=black,minimum size=0.5pt,inner sep=1.5pt]
	  \tikzstyle{empty}=[fill=none,draw=none,minimum size=0.5pt,inner sep=1.5pt]
	  \node at (30:1)[label=right:${2(r+1)}$] {};
	  \node at (60:1)[label=right:$2r$] {};
	  \node at (90:1) {};
	  \node at (120:1)[label=left:$2r$] {};
	  \node at (150:1) [label=left:${2(r+1)}$]{};
	  \node at (175:1) [label=left:$2r$]{};
	  \node[empty] at (270:1) [label=below:\phantom{t}] {};
  \end{tikzpicture}&
  \begin{tikzpicture}[scale=1]
  	  \node at (90:1.3) {$r$};
	  \path[draw,help lines] (0,0) circle (1 cm);
	  \tikzstyle{trait}=[dashed]
	  \draw[trait] (-175:1.1) arc (-175:20:1.1cm);
	  \tikzstyle{every node}=[shape=circle,fill=white,draw=black,minimum size=0.5pt,inner sep=1.5pt]
	  \tikzstyle{empty}=[fill=none,draw=none,minimum size=0.5pt,inner sep=1.5pt]
	  \node at (30:1)[label=right:$2r$] {};
	  \node at (60:1)[label=right:${2(r+1)}$] {};
	  \node at (90:1) {};
	  \node at (120:1)[label=left:${2(r+1)}$] {};
	  \node at (150:1) [label=left:$2r$]{};
	  \node at (175:1) [label=left:$2r$]{};
	  \node[empty] at (270:1) [label=below:\phantom{t}] {};
  \end{tikzpicture}\\\hline
&Type1mod: $p=2r+1$ &Type3mod: $p=2r+1$\\
  \begin{tikzpicture}[scale=1]
  	\node at (0,0) {\rotatebox{90}{Colouring 4}};
\end{tikzpicture} & \begin{tikzpicture}[scale=1]
	  \node at (90:1.3) {${r+1}$};
	  \node at (247:1.4) {${r+1}$};
	  \node at (293:1.4) {${r+1}$}; 
	  \path[draw,help lines] (0,0) circle (1 cm);
	  \tikzstyle{trait}=[dashed]
	  \draw[trait] (165:1.1) arc (165:210:1.1cm);
	  \draw[trait] (-30:1.1) arc (-30:15:1.1cm);
	  \tikzstyle{every node}=[shape=circle,fill=white,draw=black,minimum size=0.5pt,inner sep=1.5pt]
	  \node at (30:1)[label=right:${r+1}$] {};
	  \node at (60:1)[label=right:$r$] {};
	  \node at (90:1) {};
	  \node at (120:1)[label=left:$r$] {};
	  \node at (150:1) [label=left:${r+1}$]{};
	  \node at (225:1) [label=left:$r$]{};
	  \node at (255:1) {};
	  \node at (285:1) {};
	  \node at (315:1) [label=right:$r$]{};
  \end{tikzpicture}
 &\begin{tikzpicture}[scale=1]
	  \node at (90:1.3) {$r$};
	  \node at (247:1.4) {${r+1}$};
	  \node at (293:1.4) {${r+1}$};    
	  \path[draw,help lines] (0,0) circle (1 cm);
	  \tikzstyle{trait}=[dashed]
	  \draw[trait] (165:1.1) arc (165:240:1.1cm);
	  \draw[trait] (-60:1.1) arc (-60:15:1.1cm);
	  \tikzstyle{every node}=[shape=circle,fill=white,draw=black,minimum size=0.5pt,inner sep=1.5pt]
	  \node at (30:1)[label=right:$r$] {};
	  \node at (60:1)[label=right:${r+1}$] {};
	  \node at (90:1) {};
	  \node at (120:1)[label=left:${r+1}$] {};
	  \node at (150:1) [label=left:$r$]{};
	  \node at (255:1) {};
	  \node at (285:1) {};
  \end{tikzpicture}
  \end{tabular}
  \end{center}
  \caption{Weighted cycles $\mathcal{C}_p$ corresponding to the colourings 1--4.}\label{fig_colourings}
\end{table}

\medskip

Consider now the colourings 1--5 given in Theorem~\ref{thm_Axe}. For each kind of colouring, we fold the projection of $B_r(\mathbf{y})$ on a cycle $\mathcal{C}_p$, with $p\in\{2,3,r,r+1,2r,2r+1,2r+2\}$, according to the parity of $r$ (see Table~\ref{fig_colourings}). Then we use Theorem~\ref{thm:types} to give the possible values of the constant weighted sums $a$ and $b$.

The colouring 1 is $p$-periodic of odd period $p\in\{r,r+1\}$. Hence it gives two different weighted cycles. If $r$ is even, then $B_r(\mathbf{y})$ is projected and folded on the cycle $\mathcal{C}_{r+1}$ of Type 0 with $z=r+1$ and $x=2r+1$. The corresponding values of the constants are then
$$a=r+1+\alpha(2r+1)\text{ and }b=(\alpha+1)(2r+1)$$
with $\alpha\in\{0,\ldots,r-1\}$. If $r$ is  odd, $B_r(\mathbf{y})$ is projected and folded  on the cycle $\mathcal{C}_{r}$ of Type 0 with $z=3r+2$  and $x=2r+1$. So the corresponding values of the constants are $$a=3r+2+\alpha(2r+1)\text{ and }b=(\alpha+1)(2r+1)$$ with $\alpha\in\{0,\ldots,r-2\}$.

The colouring 2 is a $p$-anti-periodic colouring with $p\in\{r,r+1\}$ and $p$ even. It gives then two different weighted cycles with $2p$ vertices. If $r$ is even,  $B_r(\mathbf{y})$ is projected and folded  on the cycle $\mathcal{C}_{2r}$ 
of Type4mod with $z=r+1=y$, $x=r$ and $t=2r+2$. Then the corresponding values of the constants are 
\begin{align*}
a&=\frac{2r}{4}r+(\frac{2r}{4}-1)(r+1)+r+1=(r+1)^2+(\frac{3r}{2}+1),\\
b&=\frac{2r}{4}r+(\frac{2r}{4}-1)(r+1)+2(r+1)=r^2+(\frac{3r}{2}+1).
\end{align*}
If $r$ is odd,  $B_r(\mathbf{y})$ is projected and folded  on the cycle $\mathcal{C}_{2r+2}$ of Type4mod with $z=r=y$, $x=r+1$ and $t=0$. The corresponding values are
\begin{align*}
a&=\frac{2r+2}{4}(r+1)+\frac{2r+2}{4}r=r^2+\frac{3r+1}{2}, \\
b&=\frac{2r+2}{4}(r+1)+(\frac{2r+2}{4}-1)r=(r+1)^2+\frac{3r+1}{2}.
\end{align*}

The colouring 3 is $p$-periodic of period $p\in\{r,r+1\}$ with $p$ even. If $r$ is even, $B_r(\mathbf{y})$ is projected and folded on the cycle $\mathcal{C}_{r}$. This cycle is a particular case of a Type2mod with $t=x$ or of a Type4mod with $t=y$, according to the value of $r\bmod{4}$. So $\mathcal{C}_r$ is represented by $z(xy)^{\frac{r-2}{2}}x$ with $z=3(r+1), x=2r$ and $y=2(r+1)$. The corresponding values of the constants are either
$a=(r+1)^2$ and $b=r^2$, or
$$a=2(\alpha+1)r +2\alpha (r+1)+3(r+1)\text{ and }b= 2(\alpha+1)(2r+1)$$
with $\alpha\in\{0,\ldots,\frac{r}{2}-2\}$.
Similarly, if $r$ is odd, $B_r(\mathbf{y})$ is projected and folded on the cycle $\mathcal{C}_{r+1}$ which is a particular case of a Type2mod or a Type4mod cycle, represented by $z(xy)^{\frac{r-1}{2}}x$ with $z=r, x=2(r+1)$ and $y=2r$. The corresponding values of the constants are either $a=r^2$ and $b=(r+1)^2$ or
$$a=2(\alpha+1)(r+1)+2\alpha r+r\text{ and }b=2(\alpha+1)(2r+1)$$
with $\alpha\in\{0,\ldots,\frac{r-1}{2}-1\}$.

The colouring 4 is $2r+1$-periodic. If $r$ is even, $B_r(\mathbf{y})$ is projected and folded on the cycle $\mathcal{C}_{2r+1}$ of Type1mod with $z=r+1=y$ and $x=r$. Such weighted cycle has a constant $2$-labelling if $2r+1\equiv0\pmod{3}$. Then the corresponding values of the constants are
\begin{align*}
a&= \frac{2r+1}{3}\cdot r+\left(\frac{2r+1}{3}-1\right)(r+1)+r+1=\frac{(2r+1)^2}{3}\\
b&=\left(\frac{2r+1}{3}-1\right)r+\left(\frac{2r+1}{3}+1\right)(r+1)=\frac{(2r+1)^2}{3}+1.
\end{align*}
If $r$ is odd, $B_r(\mathbf{y})$ is projected and folded on the cycle $\mathcal{C}_{2r+1}$ of Type3mod with $z=r=y$ and $x=r+1$. Hence, under the condition that $2r+1\equiv0\pmod{3}$, the corresponding values of the constants are 
\begin{align*}
a&= \frac{2r+1}{3}(r+1)+\left(\frac{2r+1}{3}-1\right)r+r+1=\frac{(2r+1)^2}{3}+1\\
b&=\left(\frac{2r+1}{3}+1\right)(r+1)+\left(\frac{2r+1}{3}-1\right)r=\frac{(2r+1)^2}{3}+1.
\end{align*}
Since the difference $|a-b|\le4$ in this case, we never obtain these values of $a$ and $b$.

The colouring 5 is either $2$-periodic or $3$-periodic. Hence it gives five different weighted cycles. Let $c$ be the colouring 5. If $c$ is $2$-periodic, then $B_r(\mathbf{y})$ is projected and folded on $\mathcal{C}_{2}$ 
of Type 0, represented by $zx$ with
$$\left\{\begin{array}{ll}
z=(r+1)^2, x=r^2 &\text{ for $r$ even}\\
z=r^2, x=(r+1)^2 &\text{ for $r$ odd.}
\end{array}\right.$$
So the corresponding values of the constants are $a=(r+1)^2$ and $b=r^2$ for $r$ even, and $a=r^2, b=(r+1)^2$ for $r$ odd.
If $c$ is $3$-periodic, then $B_r(\mathbf{y})$ is projected and folded on $\mathcal{C}_{3}$ of Type 0. 
In that case, straightforward analysis give the weights $z$ and $x$:
\begin{itemize}
 \item $z=\frac{2r^2+2r-1}{3}$ and $x=\frac{2r^2+2r+2}{3}$ if $r=3k+1$, 
 \item $z=\frac{2r^2+2r}{3}-2k+1$ and $x=\frac{2r^2+2r}{3}+k$ if $r=3k-1$,
 \item $z=\frac{2r^2+2r}{3}+2k+1$ and $x=\frac{2r^2+2r}{3}-k$ if $r=3k$.
\end{itemize}
The corresponding values of the constants are then given by
$$a=\alpha x+z\text{ and }b=(\alpha+1)x\text{ with }\alpha\in\{0,1\}.$$
This concludes the proof of Theorem~\ref{thm:codes}.
\end{proof}

\section{Conclusions and perspectives}


Constant $2$-labellings in weighted cycles allows us to translate the periodicity of $(r,a,b)$-codes, with $r\ge2$, of the $2$-dimensional grid. It seems that for a radius $1$, many $(1,a,b)$-codes of the multidimensional grid $\mathbb{Z}^d$ are periodic (see Theorem~\ref{thm:dorbec} and ~\cite{Dorbec--Gravier--Honkala--Mollard--2009}). It would be interesting to find the corresponding weighted graphs obtained with our projection and folding method and then to study constant $2$-labellings in these graphs.  
Also, the projection and folding method is presented in general and can be applied to linear codes. It would be interesting to consider $(r,a,b)$-codes in other types of lattices as for example, in the king lattice.

The problem of finding a constant $2$-labelling of a graph is interesting in and of itself. In Theorem~\ref{thm:types}, we only obtain a characterization of constant $2$-labellings in four types of weighted cycles. It would be interesting to consider different weighted cycles, with eventually more weights. Moreover, we could study constant $2$-labelling in graphs having a big automorphisms group, for instance, in circulant graphs or in vertex-transitive graphs. Finally, we could find a natural generalization of constant $2$-labellings into constant $k$-labellings using $k$ colours and then consider their links with distinguishing numbers and weighted codes with more than two values.


\end{document}